\documentclass[runningheads, envcountsame, a4paper]{llncs}
\usepackage[T1]{fontenc}
\usepackage[pdftex]{graphicx}
\usepackage{epstopdf}
\usepackage[hidelinks]{hyperref}
\usepackage{color}

\usepackage{amsmath, amssymb}
\usepackage{tikz}
\usetikzlibrary{shapes,calc,math,backgrounds,matrix}
\usepackage{adjustbox}
\usepackage{multirow}
\definecolor{DAcolor}{rgb}{0.1,0.1,1}

\definecolor{DAHcolor}{rgb}{0.09, 0.45, 0.27}

\newcommand{\RR}[1]{{\color{red}#1}}
\newcommand{\RA}[1]{{\color{blue}#1}}
\renewcommand{\RR}[1]{{#1}}
\renewcommand{\RA}[1]{{#1}}

\usepackage[margin=2cm]{geometry}

\usepackage{etoolbox}
\makeatletter
\let\llncs@addcontentsline\addcontentsline
\patchcmd{\maketitle}{\addcontentsline}{\llncs@addcontentsline}{}{}
\patchcmd{\maketitle}{\addcontentsline}{\llncs@addcontentsline}{}{}
\patchcmd{\maketitle}{\addcontentsline}{\llncs@addcontentsline}{}{}
\makeatother
\usepackage{bookmark}

\begin{document}
\title{On Reconfiguration Graphs of Independent Sets under Token Sliding}
\author{David~Avis\inst{1,2} \and
Duc~A.~Hoang\inst{1}%
}
\authorrunning{D.~Avis and D.A.~Hoang}
\institute{Graduate School of Informatics, Kyoto University, Japan\\
\and
School of Computer Science, McGill University, Canada\\
\email{avis@cs.mcgill.ca}\\
\email{hoang.duc.8r@kyoto-u.ac.jp}%
}
\maketitle              %
\begin{abstract}
An independent set of a graph $G$ is a vertex subset $I$ such that there is no edge joining any two vertices in $I$.
Imagine that a token is placed on each vertex of an independent set of $G$. 
The $\mathsf{TS}$- ($\mathsf{TS}_k$-) reconfiguration graph of $G$ takes all non-empty independent sets (of size $k$) as its nodes, where $k$ is some given positive integer. 
Two nodes are adjacent if one can be obtained from the other by sliding a token on some vertex to one of its unoccupied neighbors.
This paper focuses on the structure and realizability of these reconfiguration graphs.
More precisely, we study two main questions for a given graph $G$: 
(1)~Whether the $\mathsf{TS}_k$-reconfiguration graph of $G$ belongs to some graph class $\mathcal{G}$ (including complete graphs, paths, cycles, complete bipartite graphs, connected split graphs, maximal outerplanar graphs, and complete graphs minus one edge)
and
(2)~If $G$ satisfies some property $\mathcal{P}$ (including $s$-partitedness, planarity, Eulerianity, girth, and the clique's size), whether the corresponding $\mathsf{TS}$- ($\mathsf{TS}_k$-) reconfiguration graph of $G$ also satisfies $\mathcal{P}$, and vice versa.
Additionally, we give a decomposition result for splitting a $\mathsf{TS}_k$-reconfiguration graph into smaller pieces.

\end{abstract}

\section{Introduction}
\label{sec:introduction}

\textit{Reconfiguration problems} arise when we want to study the relationship between \textit{solutions} of some given computational problem (called a \textit{source problem}) such as \textsc{Satisfiability}, \textsc{Independent Set}, \textsc{Dominating Set}, \textsc{Vertex-Coloring}, and so on.
The main goal of reconfiguration problems is to study the so-called \textit{reconfiguration graph}---a graph whose nodes are solutions and their adjacency can be defined via some given \textit{reconfiguration rule}.
A typical example is the well-known classic Rubik's cube puzzle, where each configuration of the Rubik's cube corresponds to a solution, and two configurations (solutions) are adjacent if one can be obtained from the other by rotating a face of the cube by either $90$, $180$, or $270$ degrees.
The question is whether there exists a path from an arbitrary node to the one where each face has only one color.
Reconfiguration graphs have been studied in the literature from three major viewpoints: \textit{structural properties} (connectivity, Hamiltonicity, planarity, and so on), \textit{realizability} (which graph can be realized as a certain type of reconfiguration graph), and \textit{algorithmic properties} (whether certain \RR{questions}, such as finding a (shortest) path between two given nodes, can be answered efficiently, and if so, design an algorithm to do it).
For an overview of this research area, readers are referred to the surveys~\cite{Heuvel13,MynhardtN19,Nishimura18}.

One of the classic reconfiguration rules is the so-called \textit{Token Sliding} ($\mathsf{TS}$).
Any vertex subset of a graph $G$ can be seen as a set of tokens placed on some vertices of $G$.
Under $\mathsf{TS}$, a token on some vertex $v$ can only be moved to one of $v$'s unoccupied adjacent vertices. 
Two vertex subsets are \textit{adjacent} under $\mathsf{TS}$ if one can be obtained from the other via a single $\mathsf{TS}$-move.
The \textit{token graph} $F_k(G)$~\cite{MonroyFHHUW12} is a reconfiguration graph whose nodes are size-$k$ vertex subsets of $G$ and edges are defined under $\mathsf{TS}$.
The study of $F_k(G)$ dates back to the 1990s, when Alavi, Behzad, Erd\H{o}s, and Lick~\cite{AlaviBEL91} considered several basic structural properties (e.g., regularity, bipartitedness, Eulerianity, etc.) of $F_2(G)$ (which they called the \textit{double vertex graph).
}
Other well-known reconfiguration rules are \textit{Token Jumping} ($\mathsf{TJ}$) and \textit{Token Addition/Removal} ($\mathsf{TAR}$), which respectively involve moving a token to any unoccupied vertex and adding/removing a single token to/from some unoccupied/occupied vertex.

In this paper, we take \textsc{Independent Set} as the source problem
and consider two types of reconfiguration graphs whose edges are defined under $\mathsf{TS}$: the $\mathsf{TS}_k$-reconfiguration graphs, for some given positive integer $k$, whose nodes are size-$k$ \textit{independent sets} of $G$ (i.e., vertex subsets whose members are pairwise non-adjacent, also called \textit{stable sets}) and the $\mathsf{TS}$-reconfiguration graphs whose nodes are independent
sets of arbitrary size.
We denote these graphs by $\mathsf{TS}_k(G)$ and $\mathsf{TS}(G)$, respectively.
(Similar definitions hold for $\mathsf{TJ}$.)
In particular, each $\mathsf{TS}_k(G)$ is an induced subgraph of $\mathsf{TS}(G)$ and also a subgraph of $F_k(G)$.

The focus of this paper is on purely graph theoretic
properties of $\mathsf{TS}(G)$ and $\mathsf{TS}_k(G)$ as opposed to
algorithmic properties, which have been well-studied in the literature.
In particular, the tractability/intractability of whether there is a path between two given nodes and several subsequent questions (e.g., if yes, whether a shortest one can be found efficiently; whether the statement holds for any pair of nodes; and so on) have been well-investigated for several graphs $G$.
Readers are referred to~\cite[Sections~4--5]{Nishimura18} for a quick summary of the recent results.
In particular, some structural properties regarding the connectivity and diameter of $\mathsf{TS}_k(G)$ can be derived from many of these algorithmic results~\cite{BonamyB17,BonsmaKW14,BrianskiFHM21,DemaineDFHIOOUY15,Fox-EpsteinHOU15,KaminskiMM11,KaminskiMM12}.
(See Appendix~\ref{apd:known-properties}.)
On the other hand, the realizability and structural properties of $\mathsf{TS}(G)$/$\mathsf{TS}_k(G)$ have not yet been systematically studied
and we initiate this study here.
More precisely, given a graph $G$, we provide some initial results regarding two main questions:
(1)~Whether the $\mathsf{TS}_k$-reconfiguration graph of $G$ belongs to some graph class $\mathcal{G}$
and
(2)~If $G$ satisfies some property $\mathcal{P}$, whether $\mathsf{TS}(G)$/$\mathsf{TS}_k(G)$ also satisfies $\mathcal{P}$ too, and vice versa.

The outline of this paper is as follows. In Section~\ref{sec:cg},
we give an example of
$\mathsf{TS}_k$-reconfiguration graphs used as a model for
a problem in computational geometry.
In Section~\ref{sec:preliminaries}, we define some terminology and notation that are used throughout this paper.
In Section~\ref{sec:is-G-a-TSk-graph}, for different graph classes $\mathcal{G}$ (including complete graphs, paths, cycles, complete bipartite graphs, connected split graphs, \RA{maximal outerplanar graphs, and complete graphs minus one edge}), we study whether there is a graph $G$ whose $\mathsf{TS}_k$-reconfiguration graph belongs to $\mathcal{G}$.
(See Table~\ref{table:is-G-a-TSk-graph}.)
In Section~\ref{sec:properties}, for different graph properties $\mathcal{P}$ (including $s$-partitedness, planarity, Eulerianity, girth, and the clique's size), we focus on answering the following question: if a graph $G$ satisfies $\mathcal{P}$, whether $\mathsf{TS}(G)$ ($\mathsf{TS}_k(G)$) does too, and vice versa. 
(See Table~\ref{table:TS-graph-properties}.)
In Section~\ref{sec:decompose-along-join}, we present a simple way of decomposing $\mathsf{TS}_k(G)$ into $\mathsf{TS}_k$-reconfiguration graphs of $G$'s subgraphs, provided that $G$ contains certain structure.
Finally, in Section~\ref{sec:concluding-remarks}, we summarize our results and propose some problems and directions for future study.

\section{A Geometrical Example}
\label{sec:cg}

\begin{figure}[!ht]
	\centering
	\begin{minipage}{.45\textwidth}
		\includegraphics[width=.5\linewidth,height=4.5cm,width=8cm]{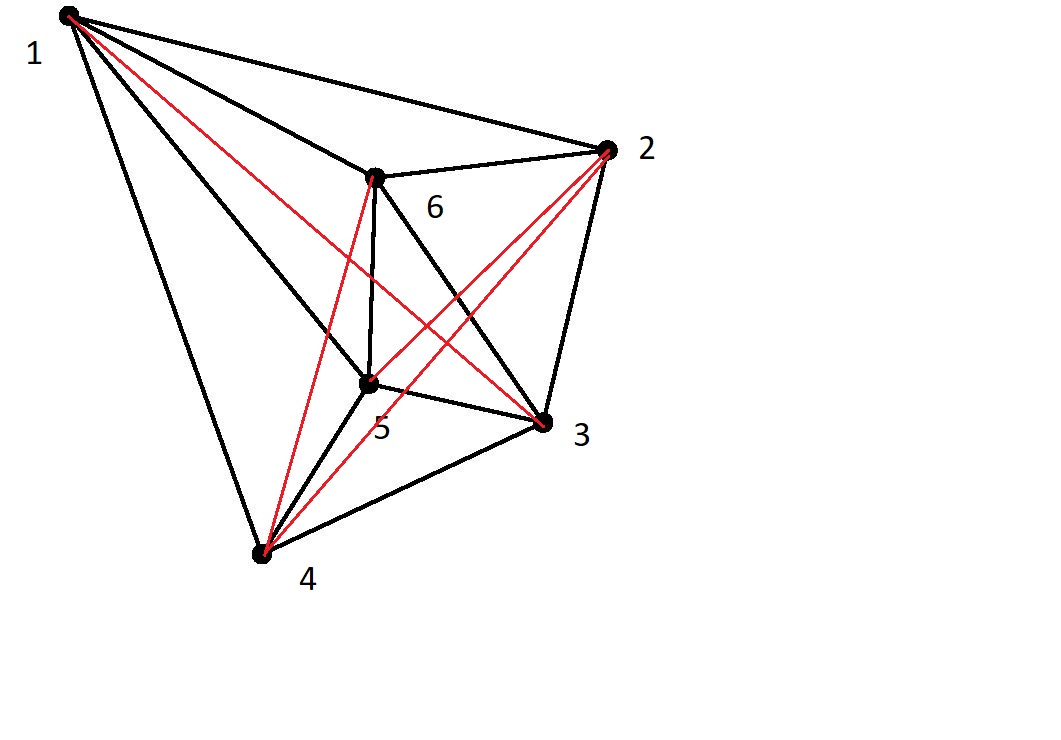}
		\caption{Point set $P$ with its line segments. A triangulation is shown by the black lines.}
		\label{fig:tr1}
	\end{minipage}\hspace*{0.5cm}%
	\begin{minipage}{.45\textwidth}
		\begin{tikzpicture}[every node/.style={circle, draw, thick, minimum size=0.1cm, fill=white}]
			\node (13) at (0,0) {$13$};
			\node (46) at (-1.5,0) {$46$};
			\node (25) at (1.5,0) {$25$};
			\node[fill=gray] (36) at (3,0) {$36$};
			\node[fill=gray] (15) at (-1.5,1.5) {$15$};
			\node[fill=gray] (56) at (0,-1.5) {$56$};
			\node (24) at (1.5,1.5) {$24$};
			\node[fill=gray] (35) at (1.5,3) {$35$};
			
			\draw[thick] (15) -- (46) -- (13) (13) -- (56) (35) -- (24) (13) -- (24) -- (36) -- (25) -- (13);
		\end{tikzpicture}
		\caption{Edge intersection graph $G$. Each number $ab$ inside a node represents an intersecting line segment of $P$. The stable set corresponds to the triangulation in \figurename~\ref{fig:tr1} is shaded grey.}
		\label{fig:tr2}
	\end{minipage}
\end{figure}
\begin{figure}[!ht]
	\centering
	\includegraphics[width=.5\linewidth,height=12cm,width=10cm]{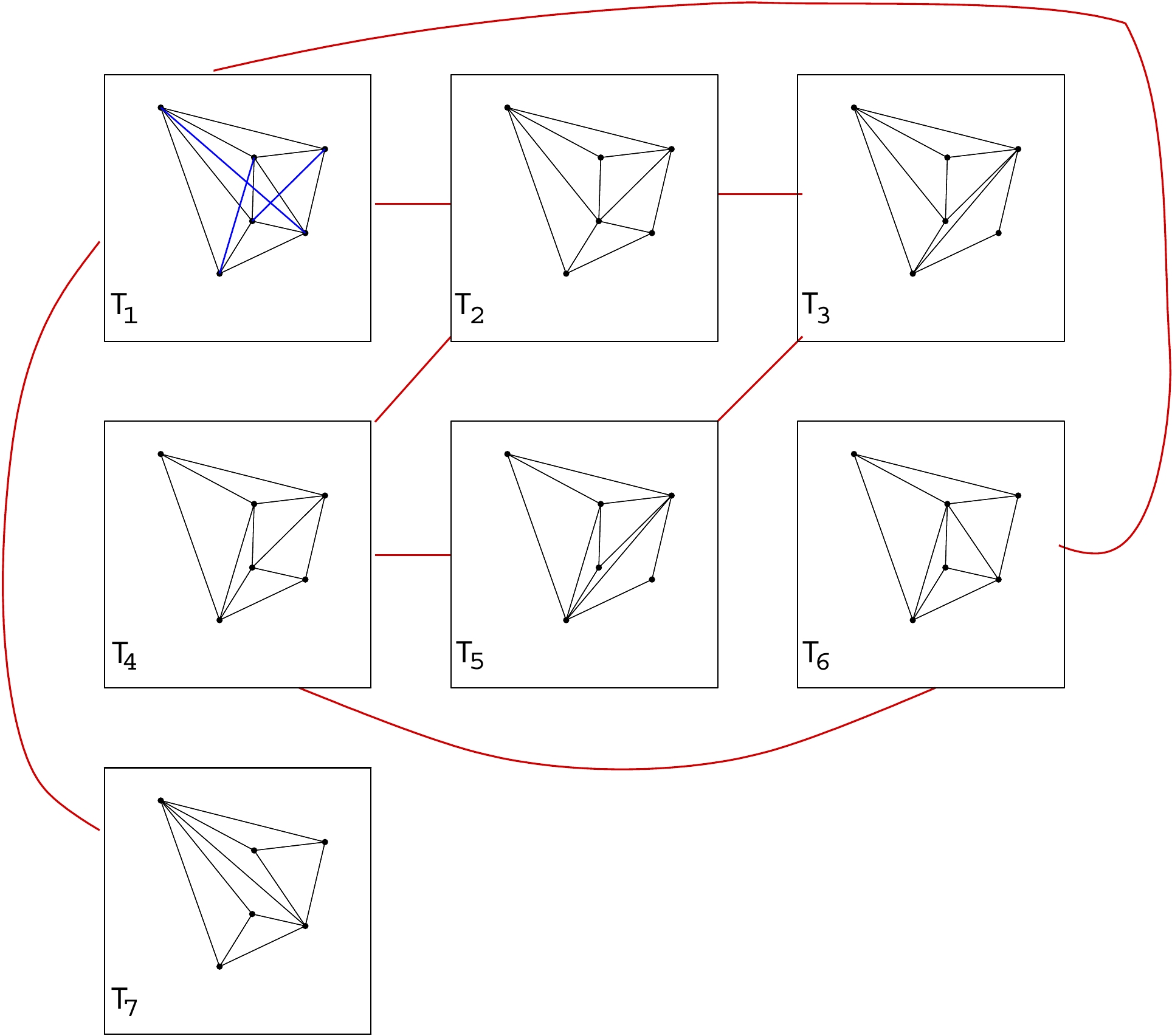}
	\caption{Triangulations of $P$. The red edges
		correspond to edge flips. In $T_1$ (the triangulation in \figurename~\ref{fig:tr1}) the three possible edge flips,
		corresponding to token slides 15-46, 56-13 and 36-25 in $G$ (\figurename~\ref{fig:tr2}), are shown in blue.}
	\label{fig:tr3}
\end{figure}

In this section we present an example of a $\mathsf{TS}$-reconfiguration graph
arising from a well studied problem in computational geometry: triangulations on a set
of planar points. 
For terminology and notation related to computational geometry which are not defined here, readers are referred to~\cite{PreparataS85}.

Given a set $P$ of $n$ points in the plane, no three collinear and no four co-circular.
Two line segments are \textit{intersecting} if they cross each other at an interior point of each segment, and \textit{non-intersecting} otherwise.
A {\it triangulation} of $P$ is any maximal set \RR{of non-intersecting segments}. 
It is well known that all triangulations have the same number of edges. 
The {\it edge intersection graph} $G$ of $P$ is the graph whose vertices $V$ are the line segments with endpoints in $P$ that intersect at least one other line segment. 
Let $L$ be the remaining set of line segments defined by $P$. Note that any edge on the convex hull of $P$ is in $L$ and there may be other such line segments. In \figurename~\ref{fig:tr1}
we give a set of $6$ points and show the $15$ segments they define. A triangulation is shown
by the solid black lines. 
The set $L$ consists of the seven segments $\{ 12, 14, 16, 23, 26, 34, 45 \} $.
We see that the edges $\{ 16, 26, 45 \}$ are internal to the convex hull.
Two vertices in $G$ are connected by an edge if their corresponding line segments
intersect.   
In \figurename~\ref{fig:tr2} we give the intersection graph $G$ for the example, which has $\alpha(G)=4$.
The triangulation appears in $G$ as the stable set
$\{ 15, 35, 36, 56 \}$.

Each vertex of $\mathsf{TS}_k(G)$ corresponds to a set of $k$ pairwise non-intersecting line segments.  
\figurename~\ref{fig:tr3} shows $\mathsf{TS}_4(G)$ for the example,
where the independent sets are shown as their corresponding
triangulations. 
The top left triangulation corresponds to the
stable set $\{15, 35, 36, 56 \}$. By token sliding we see that it is adjacent in $\mathsf{TS}_4(G)$ to stable sets $\{ 15, 35, 25, 56 \}$, 
$\{ 46, 35, 36, 56 \}$
and $\{ 15, 35, 36, 13 \} $ corresponding to the triangulations
$T_2$, $T_6$ and $T_7$. Geometrically we see that token sliding corresponds to {\it \RR{flipping}} the diagonal of a convex quadrilateral in
the triangulation.
Lawson \cite{Lawson77} showed that flipping well-chosen intersecting diagonals of
convex quadrilaterals will lead to the unique Delaunay triangulation, which proves that $\mathsf{TS}_{\alpha(G)}(G)$ is
connected. This result was generalized by Bern and Eppstein \cite{BerEpp92} to the case where a non-intersecting set of the line segments can be specified and the triangulations are 
{\it constrained} to include these segments. In this case we add these constrained segments and any segments that intersect them to $L$. 
They were able to show that the corresponding constrained Delaunay triangulation could be obtained after at most $O(n^2)$ diagonal flips,
giving a corresponding bound on the diameter of $\mathsf{TS}_{\alpha(G)}(G)$. Similar results on connectivity and
diameter for some other graph classes are given in the Appendix~\ref{apd:known-properties}.

\section{Preliminaries}
\label{sec:preliminaries}

For terminology and notation not defined here, readers are referred to~\cite{Diestel2017}.
Let $G$ be a simple, undirected graph. 
We use $V(G)$ and $E(G)$, \RR{respectively}, to denote its vertex-set and edge-set.
For two sets $I, J$, we sometimes use $I - J$ and $I + J$ to indicate $I \setminus J$ and $I \cup J$, respectively.
Additionally, if $J = \{u\}$, we sometimes write $I - u$ and $I + u$ instead of $I - \{u\}$ and $I + \{u\}$, respectively. 
The \textit{symmetric difference} of $I$ and $J$, denoted by $I \Delta J$, is simply the set $(I - J) + (J - I)$.
The \RR{\textit{neighborhood}} of a vertex $v$ in $G$, denoted by $N_G(v)$, is the set $\{w \in V(G): vw \in E(G)\}$.
The \RR{\textit{closed neighborhood}} of $v$ in $G$, denoted by $N_G[v]$, is simply the set $N_G(v) + v$.
Similarly, for a vertex subset $I \subseteq V(G)$, its \textit{\RR{neighborhood}} $N_G(I)$ and \textit{\RR{closed neighborhood}} $N_G[I]$ are respectively $\bigcup_{v \in I}N_G(v)$ and $N_G(I) + I$.
The \textit{degree} of a vertex $v$ in $G$, denoted by $\deg_G(v)$, is $\vert N_G(v) \vert$.
For a vertex subset $I$, we denote by $G[I]$ the subgraph of $G$ \textit{induced} by vertices in $I$. 
An \textit{independent set} (or \textit{stable set}) of $G$ is a vertex subset $I$ such that for any $u, v \in I$, we have $uv \notin E(G)$.
On the other hand, a \textit{clique} of $G$ is a vertex subset $K$ such that for any $u, v \in K$, we have $uv \in E(G)$.
\RR{The \textit{complement} $\overline{G}$ of $G$ is the graph whose vertices are $V(G)$ and two vertices are adjacent in $\overline{G}$ if they are not adjacent in $G$.}
We denote by $\alpha(G)$ and $\omega(G)$ the \textit{maximum} size of an independent set and a clique of $G$, respectively.
The \textit{girth} of $G$, denoted by $\text{girth}(G)$, is the smallest size of a cycle in $G$.
In case $G$ has no cycles, we define $\text{girth}(G) = \infty$, and say that it has \textit{infinite girth}, and otherwise $G$ has \textit{finite girth}. 
Two graphs $G_1$ and $G_2$ are \textit{isomorphic} if there exists a bijection $f: V(G_1) \to V(G_2)$ such that $uv \in E(G_1)$ if and only if $f(u)f(v) \in E(G_2)$.
For isomorphic graphs $G_1$ and $G_2$, we write $G_1 \simeq G_2$ or $G_1 \simeq_f G_2$ to emphasize that their isomorphism can be verified by the bijection $f$.
A graph is \textit{$H$-free} if it has no graph $H$ as an induced subgraph.

We respectively denote by $K_n$, $P_n$, and $C_n$ a \textit{complete graph}, a \textit{path}, and a \textit{cycle} on $n$ vertices.
The graph $K_3 \simeq C_3$ is also called a \textit{triangle}.
$K_n - e$ is the graph obtained from $K_n$ by removing exactly one edge.
We denote by $K_{m, n}$ a \textit{complete bipartite graph} whose partite sets are of sizes $m$ and $n$, for some positive integers $m \leq n$.
The graph $K_{1, n}$ is also called a \textit{star}.
$G$ is a \textit{split graph} if $V(G)$ can be partitioned into two sets $K$ and $S$, called a \textit{$KS$-partition} of $V(G)$, such that $K$ is a clique and $S$ is a stable set of $G$.
\RR{A \textit{$K$-max} $KS$-partition of $V(G)$ is a $KS$-partition where the size of $K$ is maximum, that is, $\vert K \vert = \omega(G)$ (e.g., see~\cite{CollinsT21}).
	It is well-known that any split graph $G$ has a unique \textit{$K$-max} $KS$-partition of $V(G)$.
	In this paper, we assume that for any split graph $G$, the $K$-max $KS$-partition of $V(G)$ is always given, and to emphasize this assumption, we write $G = (K \cup S, E)_{\text{$K$-max}}$. 
}

For a graph $G$ and a positive integer $k$, the \textit{$\mathsf{TS}_k$-reconfiguration graph} of $G$, denoted by $\mathsf{TS}_k(G)$, takes all size-$k$ independent sets of $G$ as its nodes.
Similarly, the \textit{$\mathsf{TS}$-reconfiguration graph} of $G$ takes all independent sets of $G$ as its nodes.
Two nodes (independent sets) $I$ and $J$ are \textit{adjacent} in either $\mathsf{TS}(G)$ or $\mathsf{TS}_k(G)$ if there exist $u, v \in V(G)$ such that $I - J = \{u\}$, $J - I = \{v\}$, and $uv \in E(G)$.
Naturally, we call a graph $F$ a \textit{$\mathsf{TS}_k$-reconfiguration graph} if there exists a graph $G$ such that $F \simeq \mathsf{TS}_k(G)$ or more precisely $F \simeq_f \mathsf{TS}_k(G)$ where $f$ is some bijection that can be used for verifying their isomorphism.
One can think of $f$ as a way to \textit{label vertices of $F$ by size-$k$ independent sets of $G$}.

\section{Graphs That Are (Not) $\mathsf{TS}_k$-Reconfiguration Graphs}
\label{sec:is-G-a-TSk-graph}

In this section, we study the realizability of $\mathsf{TS}_k$-reconfiguration graphs.
It is trivial that any graph $G$ is also a $\mathsf{TS}_1$-reconfiguration graph, since $G \simeq \mathsf{TS}_1(G)$.
Therefore, in this section, we always assume $k \geq 2$.

For a graph $G$, let $L_k(G)$ be the graph whose nodes are size-$k$ cliques of $G$ and two nodes are adjacent if they have exactly $k - 1$ vertices in common. 
In particular, $L_2(G)$ is also known as the \textit{line graph} of $G$.
The following lemma describes a relationship between $L_k(G)$ and the $\mathsf{TS}_k$-reconfiguration graph of its complement $\overline{G}$.
\figurename~\ref{fig:lgtf} illustrates this relationship for $k = 2$.

\begin{lemma}\label{lem:lgtf}
	Given a graph $G$.
	Then, $\mathsf{TS}_k(\overline{G})$ is a subgraph of $L_k(G)$.
	Moreover, $L_k(G) \simeq \mathsf{TS}_k(\overline{G})$ if and only if $G$ is $K_{k+1}$-free.
\end{lemma}
\begin{proof}
	By definition, for any $I = \{a_1, \dots, a_k\} \in V(\mathsf{TS}_k(\overline{G}))$, we have $a_ia_j \notin E(\overline{G})$ and therefore $a_ia_j \in E(G)$ for $1 \leq i < j \leq k$, which implies that $a_1\dots a_k \in V(L_k(G))$.
	As a result, the mapping $f: V(\mathsf{TS}_k(\overline{G})) \to V(L_k(G))$ defined by $f(\{a_1, \dots, a_k\}) = a_1\dots a_k$ is bijective.
	Moreover, if $I = \{a_1, \dots, a_k\}$ and $J = \{a_1^\prime, \dots, a_k^\prime\}$ are adjacent in $\mathsf{TS}_k(\overline{G})$, we must have $\vert I \cap J \vert = k - 1$, and therefore $a_1\dots a_k$ and $a_1^\prime\dots a_k^\prime$ are adjacent in $L_k(G)$.
	Therefore, $\mathsf{TS}_k(\overline{G})$ is a subgraph of $L_k(G)$.
	
	We now claim that $\mathsf{TS}_k(\overline{G}) \simeq_f L_k(G)$ if and only if $G$ is $K_{k+1}$-free.
	\begin{itemize}
		\item[$(\Rightarrow)$] Suppose to the contrary that $\mathsf{TS}_k(\overline{G}) \simeq_f L_k(G)$ and $G$ has a $K_{k+1} = a_1\dots a_ka_{k+1}$.
		Thus, $I = \{a_1, \dots, a_{k-1},\allowbreak a_k\}$ and $J = \{a_1, \dots, a_{k-1}, a_{k+1}\}$ are vertices of $\mathsf{TS}_k(\overline{G})$.
		Since $a_1\dots a_{k-1}a_k$ and $a_1\dots a_{k-1}a_{k+1}$ are adjacent in $L_k(G)$, we have $IJ \in E(\mathsf{TS}_k(\overline{G}))$.
		It follows that $a_ka_{k+1} \in E(\overline{G})$, which means $a_ka_{k+1} \notin E(G)$, which is a contradiction.
		Therefore, $G$ is $K_{k+1}$-free.
		\item[$(\Leftarrow)$] Suppose that $G$ is $K_{k+1}$-free.
		It suffices to show that $IJ \in E(\mathsf{TS}_k(\overline{G}))$ if and only if $f(I)f(J) \in E(L_k(G))$.
		Since $\mathsf{TS}_k(\overline{G})$ is a subgraph of $L_k(G)$, the only-if direction is clear.
		We now show the if direction.
		Without loss of generality, let $a_1\dots a_{k-1}a_k$ and $a_1\dots a_{k-1}a_{k+1}$ be two adjacent vertices in $L_k(G)$.
		Since $G$ is $K_{k+1}$-free, $a_ka_{k+1} \notin E(G)$, which means $a_ka_{k+1} \in E(\overline{G})$ and therefore $IJ \in E(\mathsf{TS}_k(\overline{G}))$, where $I = f^{-1}(a_1\dots a_{k-1}a_k)$ and $J = f^{-1}(a_1\dots a_{k-1}a_{k+1})$.
	\end{itemize}
	\qed\end{proof}
\begin{figure}[!ht]
	\centering
	\begin{tikzpicture}[every node/.style={circle, draw, thick, minimum size=0.3cm, fill=white}]
		\begin{scope}
			\node[label=above:{$v_1$}] (v1) at (0,0) {};
			\node[label=left:{$v_2$}] (v2) at (-1,-1) {};
			\node[label=left:{$v_3$}] (v3) at (-1,-2.5) {};
			\node[label=right:{$v_4$}] (v4) at (1,-2.5) {};
			\node[label=right:{$v_5$}] (v5) at (1,-1) {};
			
			\draw[thick] (v1) -- (v2) -- (v3) -- (v4) -- (v5) -- (v1) (v2) -- (v5);
			
			\node[draw=none, fill=none] (G) at (0, -3) {$G$};
		\end{scope}
		\begin{scope}[shift={(5,0)}]
			\node[label=above:{$v_1$}] (v1) at (0,0) {};
			\node[label=left:{$v_2$}] (v2) at (-1,-1) {};
			\node[label=left:{$v_3$}] (v3) at (-1,-2.5) {};
			\node[label=right:{$v_4$}] (v4) at (1,-2.5) {};
			\node[label=right:{$v_5$}] (v5) at (1,-1) {};
			
			\draw[thick] (v2) -- (v4) -- (v1) -- (v3) -- (v5);
			
			\node[draw=none, fill=none] (G) at (0, -3) {$\overline{G}$};
		\end{scope}
		\begin{scope}[shift={(0, -4)}]
			\node (12) at (-1,0) {$12$};
			\node (15) at (1,0) {$15$};
			\node (23) at (-1,-1.5) {$23$};
			\node (25) at (1,-1.5) {$25$};
			\node (34) at (-1,-3) {$34$};
			\node (45) at (1,-3) {$45$};
			
			\draw[ultra thick] (12) -- (23) -- (34) -- (45) edge[bend right=45] (15) (23) -- (25) -- (45); 
			\draw[thick] (12) -- (15) -- (25) -- (12);
			
			\node[draw=none, fill=none] (L2G) at (0, -4) {$L_2(G)$};
		\end{scope}
		\begin{scope}[shift={(5, -4)}]
			\node (12) at (-1,0) {$12$};
			\node (15) at (1,0) {$15$};
			\node (23) at (-1,-1.5) {$23$};
			\node (25) at (1,-1.5) {$25$};
			\node (34) at (-1,-3) {$34$};
			\node (45) at (1,-3) {$45$};
			
			\draw[thick] (12) -- (23) -- (34) -- (45) edge[bend right=45] (15) (23) -- (25) -- (45); 
			
			\node[draw=none, fill=none] (L2G) at (0, -4) {$\mathsf{TS}_2(\overline{G})$};
		\end{scope}
	\end{tikzpicture}
	\caption{A graph $G$, its complement $\overline{G}$, its line graph $L_2(G)$, and the graph $\mathsf{TS}_2(\overline{G})$. Each number of the form $ab$ inside a node represents a vertex subset $\{v_a, v_b\}$ which forms both a node in $L_2(G)$ and an independent set of $\overline{G}$.}
	\label{fig:lgtf}
\end{figure}
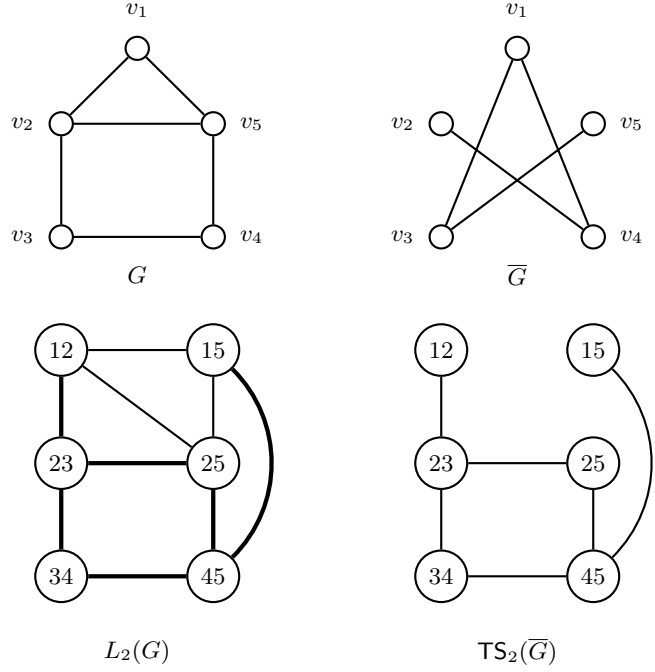
In \figurename~\ref{fig:lgtf} we see that $G$ contains a triangle and indeed 
$L_2(G)$ is not isomorphic to $ \mathsf{TS}_2(\overline{G})$. 
If we break the triangle by deleting the edge
$v_1v_2$ this corresponds to deleting the vertex $12$ in $L_2(G)$ and 
$\mathsf{TS}_2(\overline{G})$ along with adding the edge from $15$ to $25$. We now see that
$L_2(G) \simeq \mathsf{TS}_2(\overline{G})$.

In the rest of this section, for some graph class $\mathcal{G}$, we answer the following question: Does there exist a graph $G$ such that $\mathsf{TS}_k(G) \in \mathcal{G}$, \RR{for some $k \geq 2$}?
Our results are described in Table~\ref{table:is-G-a-TSk-graph}.

\begin{table}[!ht]
	\centering
	\caption{Does $G$ such that $\mathsf{TS}_k(G) \in \mathcal{G}$ exist ($k \geq 2$)?}
	\label{table:is-G-a-TSk-graph}
	\begin{tabular}{|c|c|c|}
		\hline
		$\mathcal{G}$ & Does $G$ exist? & Ref. \\
		\hline
		\multirow{3}{*}{$K_n$} & yes & \multirow{3}{*}{Cor.~\ref{cor:Kn-is-a-TSk-graph}}\\
		& $\vert V(G) \vert = n + k - 1$ & \\
		& $\vert E(G) \vert %
		= n(n-1)/2$ & \\
		\hline
		\multirow{3}{*}{$P_n$} & yes & \multirow{6}{*}{\RA{Cor.~\ref{cor:PnCn-are-TSk-graphs}}}\\
		& $\vert V(G) \vert = (n + 1) + k - 2 = n + k - 1$ & \\
		& $\vert E(G) \vert %
		= n(n-1)/2$ & \\
		\cline{1-2}
		\multirow{3}{*}{$C_n$} & yes & %
		\\
		& $\vert V(G) \vert = n + k - 2$ & \\
		& $\vert E(G) \vert %
		= n(n-3)/2$ & \\
		\hline
		\multirow{3}{*}{$K_{m, n}$ ($m \leq n$)} & yes, iff $m = 1$ and $n \leq k$ or $m = n = 2$ & \multirow{3}{*}{Prop.~\ref{prop:Kmn-isnt-a-TSk-graph}} \\
		& $\vert V(G) \vert = n + k$ or $\vert V(G) \vert = k + 2$ & \\
		& $\vert E(G) \vert = n(n+1)/2$ or $\vert E(G) \vert = 2$ & \\
		\hline
		\multirow{4}{*}{\begin{tabular}{@{}c@{}}connected\\ $F = (K \cup S, E)_{\text{$K$-max}}$\end{tabular} } & \begin{tabular}{@{}c@{}}yes, iff $\vert N_F(v) \cap S \vert \leq k-1$ and $\vert N_F(w) \vert = 1$\\ for every $v \in K$ and $w \in S$\end{tabular} & \multirow{4}{*}{Prop.~\ref{prop:conn-split-isnt-a-TSk-graph}}\\
		& $\vert V(G) \vert = \vert K \vert + \vert S \vert + k - 1$ & \\
		& \begin{tabular}{@{}c@{}}$\vert E(G) \vert = \binom{\vert K \vert}{2} + \binom{\vert S \vert}{2} + \sum_{v\in K}\vert N_F(v) \cap S \vert +$\\ $+ \sum_{\substack{v\in K\\ \vert N_F(v) \cap S \vert \neq 0}}(\vert K \vert - 1)$\end{tabular} & \\
		\hline
		\RA{maximal outerplanar} & \multirow{2}{*}{\RA{yes, iff $n \leq 3$}} & \multirow{2}{*}{\RA{Prop.~\ref{prop:general}}}\\
		\cline{1-1}
		\RA{$K_n - e$} & & \\
		\hline
	\end{tabular}
\end{table}

We now prove some useful observations.
\begin{proposition}\label{prop:induced}
	If $H$ is an induced subgraph of $G$, then $\mathsf{TS}_k(H)$ is an induced subgraph of $\mathsf{TS}_k(G)$.
	\RA{The reverse does not hold for any $k \geq 2$.}
\end{proposition}
\begin{proof}
	Suppose that there exist $I, J \in V(\mathsf{TS}_k(H))$ such that $IJ \in E(\mathsf{TS}_k(G))$.
	Since $IJ \in E(\mathsf{TS}_k(G))$, there exists $u, v \in V(G)$ such that $I - J = \{u\}$, $J - I = \{v\}$, and $uv \in E(G)$.
	Since $I, J \in V(\mathsf{TS}_k(H))$, it follows that $u, v \in V(H)$.
	As $H$ is an induced subgraph of $G$ and $uv \in E(G)$, we must have $uv \in E(H)$.
	Therefore, $IJ \in E(\mathsf{TS}_k(H))$.
	
	\RA{Now, if $H = C_{2k}$ and $G = K_{1,k+1}$, one can readily verify that $\mathsf{TS}_k(H)$ is an induced subgraph of $\mathsf{TS}_k(G)$ but $H$ is clearly not an induced subgraph of $G$.}
	\qed\end{proof}

\begin{proposition}\label{prop:TS-max-reconf-graph}
	Given a graph $H$, let $G = \mathsf{TS}_{\alpha(H)}(H)$.
	Then, for every $k \geq \alpha(H)$, $G$ is a $\mathsf{TS}_k$-reconfiguration graph.
	More precisely, there exists a graph $H^\prime$ having $\vert V(H) \vert + k - \alpha(H)$ vertices and $\vert E(H) \vert$ edges such that $G \simeq \mathsf{TS}_k(H^\prime)$.
\end{proposition}
\begin{proof}
	Let $H^\prime$ be the graph obtained by adding a set $X$ of $(k - \alpha(H))$ new vertices to $H$.
	Any size-$k$ independent set in $H^\prime$ is a disjoint union of $X$ and a maximum independent set in $H$.
	Then, $G = \mathsf{TS}_{\alpha(H)}(H) \simeq \mathsf{TS}_k(H^\prime)$.
	\qed\end{proof}

A direct consequence of Proposition~\ref{prop:TS-max-reconf-graph} when $H = K_n$ is as follows. 
\begin{corollary}\label{cor:Kn-is-a-TSk-graph}
	$K_n$ is a $\mathsf{TS}_k$-reconfiguration graph, for any integers $k \geq 2$ and $n \geq 2$.
\end{corollary}

\RA{
	A direct consequence of Lemma~\ref{lem:lgtf} and Proposition~\ref{prop:TS-max-reconf-graph} is as follows.
	\begin{corollary}\label{cor:PnCn-are-TSk-graphs}
		\begin{itemize}
			\item[(a)] $P_n$ is a $\mathsf{TS}_k$-reconfiguration graph, for any integers $k \geq 2$ and $n \geq 1$.
			
			\item[(b)] $C_n$ is a $\mathsf{TS}_k$-reconfiguration graph, for any integers $k \geq 2$ and $n \geq 3$.
		\end{itemize}
	\end{corollary}
	\begin{proof}
		\begin{itemize}
			\item[(a)] Since $P_n$ is triangle-free, from Lemma~\ref{lem:lgtf} we have 
			\[
			P_n \simeq L_2(P_{n+1}) \simeq \mathsf{TS}_2(\overline{P_{n+1}}).
			\]
			This settles the case $k = 2$.
			For $k = 3$, note that $\alpha(\overline{P_{n+1}}) = \omega(P_{n+1}) = 2$, therefore it follows from Proposition~\ref{prop:TS-max-reconf-graph} that $P_n$ is a $\mathsf{TS}_k$-reconfiguration graph for every $k \geq 2$.
			
			\item[(b)] Since $C_3 \simeq K_3$, Corollary~\ref{cor:Kn-is-a-TSk-graph} settles the case $n = 3$.
			For $n \geq 4$, again, since $C_n$ ($n \geq 4$) is triangle-free, from Lemma~\ref{lem:lgtf} we have 
			\[
			C_n \simeq L_2(C_{n}) \simeq \mathsf{TS}_2(\overline{C_{n}}).
			\]
			This settles the case $k = 2$.
			Again, for $k = 3$, note that $\alpha(\overline{C_n}) = \omega(C_n) = 2$, therefore it follows from Proposition~\ref{prop:TS-max-reconf-graph} that $C_n$ ($n \geq 4$) is a $\mathsf{TS}_k$-reconfiguration graph for every $k \geq 2$.
		\end{itemize}
	\end{proof}
}

Under certain conditions, one can construct a new $\mathsf{TS}_k$-reconfiguration graph from a known one, as in the following proposition.
\begin{proposition}\label{prop:add-vG}
	Let $I$ be an independent set of a given graph $G$.
	Let $k = \vert I \vert + 1$.
	Let $G^\prime$ be the graph obtained from $G$ by adding a new vertex $v_G$ and joining it to every vertex in $V(G) - I$.
	Then, $\mathsf{TS}_k(G^\prime)$ is obtained from $\mathsf{TS}_k(G)$ by adding a new node $I + v_G$ and all incident edges.
\end{proposition}
\begin{proof}
	By definition, $G$ is an induced subgraph of $G^\prime$, and therefore by Proposition~\ref{prop:induced}, $\mathsf{TS}_k(G)$ is an induced subgraph of $\mathsf{TS}_k(G^\prime)$.
	It suffices to show that $I + v_G$ is the unique node in $V(\mathsf{TS}_k(G^\prime)) - V(\mathsf{TS}_k(G))$.
	Observe that for any $J \in V(\mathsf{TS}_k(G^\prime)) - V(\mathsf{TS}_k(G))$, we must have $v_G \in J$.
	Now, if $J \neq I + v_G$, there must be some $w \in V(G) - I$ such that $w \in J$.
	However, by definition of $G^\prime$, we have $wv_G \in E(G^\prime)$, which contradicts $\{w, v_G\} \subseteq J \in V(\mathsf{TS}_k(G^\prime))$.
	Therefore, $J = I + v_G$. 
	Our proof is complete.
	\qed\end{proof}

\begin{proposition}\label{prop:Kmn-isnt-a-TSk-graph}
$K_{m, n}$ is a $\mathsf{TS}_k$-reconfiguration graph for some integers $k \geq 2$ and $n \geq m \geq 1$ if and only if $m = 1$ and $n \leq k$ or $m = n = 2$.
\end{proposition}
\begin{proof}
\begin{itemize}
\item[($\Leftarrow$)] Corollary~\ref{cor:PnCn-are-TSk-graphs} settles the case $m = n = 2$, since $K_{2,2} \simeq C_4$.
It remains to consider the case $m = 1$ and $n \leq k$.
In this case, we claim that there exists a graph $G$ such that $\mathsf{TS}_k(G) \simeq K_{1,n}$.
Let $I_1 = \{a_1, \dots, a_k\}$ be an independent set of size $k$, and let $K_n$ be a clique whose vertices are $b_1, \dots, b_n$. 
We construct $G$ by joining each $a_i$ with $b_i$, for $1 \leq i \leq n$.

It remains to show that $\mathsf{TS}_k(G) \simeq K_{1,n}$.
Note that $V(\mathsf{TS}_k(G)) = I_1 \cup \bigcup_{i=1}^nJ_i$, where $J_i = I_1 - a_i + b_i$ for $1 \leq i \leq n$.
It follows that $\mathsf{TS}_k(G)$ has exactly $n+1$ vertices.
Additionally, from the construction of $G$, the set $I_1$ is adjacent to every $J_i$ ($1 \leq i \leq n$) in $\mathsf{TS}_k(G)$, and for every $1 \leq i < j \leq n$, we always have $\vert J_j - J_i \vert = \vert J_i - J_j \vert = 2$, which implies that $J_iJ_j \notin E(\mathsf{TS}_k(G))$.
Therefore, $\mathsf{TS}_k(G) \simeq K_{1,n}$.

\item[($\Rightarrow$)] 
We first show that if $m = 1$ and $n \geq k+1$, there does not exist any $G$ such that $\mathsf{TS}_k(G) \simeq K_{m,n}$, where $k \geq 2$.
Suppose to the contrary that $n \geq k+1$ and $G$ exists. 
Let $V(\mathsf{TS}_k(G)) = \{I_1, J_1, \dots, J_n\}$ and assume without loss of generality that $I_1$ is adjacent to $J_i$ in $\mathsf{TS}_k(G)$ for $1 \leq i \leq n$.
Since $n \geq k+1$, by the pigeonhole principle, there must be some vertex $u \in I_1$ such that sliding the token on $u$ results at least two different size-$k$ independent sets of $G$, say $J_1$ and $J_2$, that are both adjacent to $I_1$.
Then, we can write $I_1 = I + x + u$, $J_1 = I + x + v$, and $J_2 = I + x + w$, for some size-$(k-2)$ independent set $I$ of $G$ such that none of the distinct vertices $u, v, w, x$ is in $I$.
By definitions of $J_1$ and $J_2$, both $v$ and $w$ are not in $I_1$.
Now, let $J_3 = I + v + w$.
One can verify that $J_3 \in V(\mathsf{TS}_k(G))$ and therefore must be adjacent to $I_1$ in $\mathsf{TS}_k(G)$, which implies $\vert J_3 - J_1 \vert = 1$.
However, note that $\{v, w\} \subseteq J_3 - I_1$, which is a contradiction.

\begin{figure}[!ht]
	\centering
	\begin{tikzpicture}[every node/.style={draw, thick, circle, minimum width=0.3cm, fill=white, transform shape}, scale=0.7]
		\begin{scope}
			\node [label=above:$I_1$] (I1) at (0,0) {$I + x + a_1$};
			\node [label=above:$J_1$] (J1) at (3,0) {$I + x + b_1$};
			\node [label=below:$I_2$] (I2) at (0,-3) {$I + x + a_2$};
			\node [label=below:$J_2$] (J2) at (3,-3) {$I + x + b_2$};
			
			\draw[thick] (I1) -- (J1) -- (I2) -- (J2) -- (I1);
			\node[draw=none] at (1.5, -4.5) {(a)};
		\end{scope}
		\begin{scope}[xshift={7cm}]
			\node [label=above:$I_1$] (I1) at (0,0) {$I + a_1 + a_2$};
			\node [label=above:$J_1$] (J1) at (3,0) {$I + b_1 + a_2$};
			\node [label=below:$I_2$] (I2) at (0,-3) {$I + b_1 + b_2$};
			\node [label=below:$J_2$] (J2) at (3,-3) {$I + a_1 + b_2$};
			
			\draw[thick] (I1) -- (J1) -- (I2) -- (J2) -- (I1);
			\node[draw=none] at (1.5, -4.5) {(b)};
		\end{scope}
	\end{tikzpicture}
	\caption{Two possible forms of the $4$-cycle $I_1J_1I_2J_2$ in the proof of Proposition~\ref{prop:Kmn-isnt-a-TSk-graph}. Here $I$ is a size-$(k-2)$ independent set of a graph $G$.}
	\label{fig:Kmn-isnt-a-TSk-graph}
\end{figure}
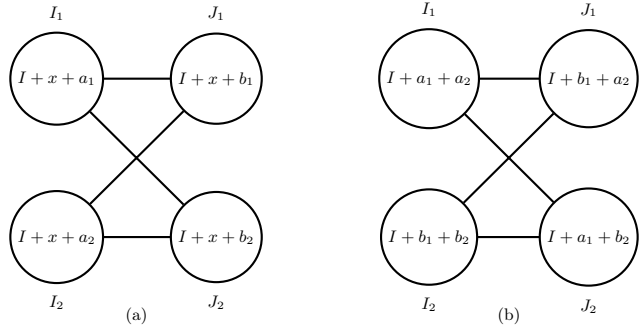

It remains to show that if $m \geq 2$ and $n > 2$, there does not exist any $G$ such that $\mathsf{TS}_k(G) \simeq K_{m,n}$.
Again, suppose to the contrary that $m \geq 2$, $n > 2$, and $G$ exists.
Suppose that $\mathsf{TS}_k(G)$ has the partite sets $X = \{I_1, \dots, I_m\}$ and $Y = \{J_1, \dots, J_n\}$, where $I_i$ ($1 \leq i \leq m$) and $J_j$ ($1 \leq j \leq n$) are size-$k$ independent sets of $G$.
Let consider the length-$4$ cycle $I_1J_1I_2J_2$, and suppose that we initially slide tokens in $I_1$.
To form such a cycle, at most two tokens in $I_1$ can be moved from their original positions, otherwise we need to perform more than four token-slides to obtain a cycle.
More formally, for some size-$(k-2)$ independent set $I$ of $G$, the sets $I_1, J_1, I_2, J_2$ can only be in one of the following two forms (see \figurename~\ref{fig:Kmn-isnt-a-TSk-graph}):
\begin{itemize}
	\item[(a)] $I_i = I + x + a_i$ and $J_i = I + x + b_i$ ($1 \leq i \leq 2$). 
	Intuitively, this corresponds to sliding a single token in $I_1$ along the cycle $a_1b_1a_2b_2$ of $G$.
	\item[(b)] $I_1 = I + a_1 + a_2$, $J_1 = I + b_1 + a_2$, $I_2 = I + b_1 + b_2$, and $J_2 = I + a_1 + b_2$.
	Intuitively, this corresponds to sliding tokens ``back and forth'' along the edges $a_1b_1$ and $a_2b_2$ of $G$. 
\end{itemize}
In both cases, none of $x$, $a_i$ and $b_i$ is in $I$, for $1 \leq i \leq 2$.
It remains to show that both cases lead to some contradiction.
Now, if (a) happens, let $J = I + a_1 + a_2$.
One can verify that $J \in \mathsf{TS}_k(G)$.
By definition of $J$, it cannot be adjacent to $I_1$ in $\mathsf{TS}_k(G)$, otherwise $xa_2 \in E(G)$, which contradicts $I_2 = I + x + a_2 \in V(\mathsf{TS}_k(G))$.
Additionally, $J$ cannot be adjacent to $J_1$, because $\{a_1, a_2\} \subseteq J - J_1$.
However, since $\mathsf{TS}_k(G) \simeq K_{m,n}$, it follows that $J$ must be adjacent to either $I_1$ or $J_1$, which is a contradiction.

It remains to consider the case (b) happens.
In this case, let $J \in \mathsf{TS}_k(G)$ be such that $J$ is adjacent to both $I_1$ and $I_2$ and $J \notin \{J_1, J_2\}$.
Since $n > 2$, such a set $J$ exists.
Now, if $\{a_1, a_2\} \subseteq J$, we also have $\{a_1, a_2\} \subseteq J - I_2 = J - (I + b_1 + b_2)$.
(Both $a_1$ and $a_2$ are not in $I$.)	
This contradicts the adjacency of $J$ and $I_2$.
Therefore, $\{a_1, a_2\} \nsubseteq J$, and similarly, so does $\{b_1, b_2\}$.
Since $J$ is adjacent to both $I_1$ and $I_2$, it contains at least one member of $\{a_1, a_2\}$ and $\{b_1, b_2\}$, respectively.
It follows that either $\{a_1, b_2\}$ or $\{a_2, b_1\}$ is in $J$.
Now, if $\{a_1, b_2\} \subseteq J$, it follows that $J = I + a_1 + x = I + b_2 + y$ for some $x, y \in V(G)$ such that $xa_2, yb_1 \in E(G)$.
Then, it follows that $x = b_2$ and $y = a_1$, that is, $J = I + a_1 + b_2 = J_1$, a contradiction.
The case $\{a_2, b_1\} \subseteq J$ can be showed similarly.
Our proof is complete.
\end{itemize}
\qed\end{proof}

\begin{lemma}\label{lem:Kn-subgraph}
Given a graph $G$.
Then, if $\mathsf{TS}_k(G)$ has a $K_n$, so does $G$, for integers $k \geq 2$ and $n \geq 3$.
\end{lemma}	
\begin{proof}
We first consider the case $n=3$, fix any $k \ge 2$ and suppose that $\mathsf{TS}_k(G)$
has a triangle labelled by independent sets $I_1, I_2, I_3$. Without loss of generality we may
assume $I_1 = \{a_1, \dots, a_{k-1}, w\}$ and $I_2 = \{a_1, \dots, a_{k-1}, x\}$, where $wx \in E(G)$. 
Now, there are two possibilities for $I_3$: either (1) $I_3 = \{a_1, \dots, a_{k-1}, y\}$ where $xy \in E(G)$ or (2) $I_3 = \{a_1, \dots, a_{i-1}, y, a_{i+1}, \dots, a_{k-1}, x\}$ where $a_iy \in E(G)$ for some $i \in \{1, \dots, k-1\}$.
Since $I_1$ and $I_3$ are adjacent in $\mathsf{TS}_k(G)$, only (1) can happen.
Therefore, we must have $yw \in E(G)$ and so $w,x,y$ define a triangle in $G$.
\qed\end{proof}

To conclude this section, we show that a connected split graph is a $\mathsf{TS}_k$-reconfiguration graph if and only if it satisfies certain restricted conditions, as described in the following lemma.

\begin{proposition}\label{prop:conn-split-isnt-a-TSk-graph}
Fix $k \geq 2$. A connected split graph $F = (K \cup S, E)_{\text{$K$-max}}$ is a $\mathsf{TS}_k$-reconfiguration graph if and only if $\vert N_F(v) \cap S \vert \leq k-1$ and $\vert N_F(w) \vert = 1$ for every $v \in K$ and $w \in S$.
\end{proposition}
\begin{proof}
\begin{itemize}
\item[$(\Leftarrow)$]
For a connected split graph $F = (K \cup S, E)_{\text{$K$-max}}$ suppose that $\vert N_F(v) \cap S \vert \leq k-1$ and $\vert N_F(w) \vert = 1$ for every $v \in K$, $w \in S$.
Additionally, let $m = \vert K \vert$ and $n = \vert S \vert$.
Suppose that $K = \{v_1, \dots, v_m\}$ and let $n_i = \vert N_F(v_i) \cap S \vert = \vert \{w^i_1, \dots, w^i_{n_i}\} \vert$ for $1 \leq i \leq m$.
Observe that $n = \sum_{i=1}^mn_i$ and by assumption $n_i \leq k-1$ for every $1 \leq i \leq m$.
We construct a graph $G$ such that $\mathsf{TS}_k(G) \simeq F$ as follows.
(See \figurename~\ref{fig:split-TSk}.)
\begin{itemize}
	\item Let $I = \{a_1, \dots, a_{k-1}\}$ be an independent set.
	Let $K_m$ be a size-$m$ complete graph with $V(K_m) = \{b_1, \dots, b_m\}$.
	Let $K_n$ be a size-$n$ complete graph with $V(K_n) = \bigcup_{i=1}^m\bigcup_{j=1}^{n_i}\{x^i_j\}$.
	\item The graph $G$ has $V(G) = I \cup V(K_n) \cup V(K_m)$.
	The edges of $G$ are defined by joining each $x^i_j \in V(K_n)$ to $a_j \in I$ ($1 \leq j \leq n_i$) and every vertex in $V(K_m) - b_i$.
\end{itemize}
\begin{figure}[!ht]
	\centering
		\begin{tikzpicture}[scale=0.75, every node/.style={circle, draw, thick, minimum size=0.5cm, fill=white, transform shape}]
			\begin{scope}[shift={(0,0.5)}]
				\node[label=left:{\large $a_1a_2\dots a_{k-1}b_1$}] (k1) at (0,0) {};
				\node[label=left:{\large $a_1a_2\dots a_{k-1}b_2$}] (k2) at (0,-1) {};
				\node[label=left:{\large $a_1a_2\dots a_{k-1}b_m$}] (km) at (0,-3) {};
				
				\node[label=right:{\large $x^1_1a_2\dots a_{n_1}\dots b_1$}] (x11) at (2,1) {};
				\node[label=right:{\large $a_1x^1_2\dots a_{n_1}\dots b_1$}] (x12) at (2,0) {};
				\node[label=right:{\large $a_1a_2\dots x^1_{n_1}\dots b_1$}] (x1n1) at (2,-1.5) {};
				
				\node[label=right:{\large $x^m_1a_2\dots a_{n_m}\dots b_m$}] (xm1) at (2,-2.5) {};
				\node[label=right:{\large $a_1a_2\dots x^m_{n_m}\dots b_m$}] (xmnm) at (2,-3.5) {};
				
				\draw[thick] (k1) -- (x11) (k1) -- (x12) (k1) -- (x1n1) (km) -- (xm1) (km) -- (xmnm);
				\draw[thick, dotted] ([yshift={-0.1cm}]k2.south) -- ([yshift={0.1cm}]km.north) ([yshift={-0.1cm}]x12.south) -- ([yshift={0.1cm}]x1n1.north) ([yshift={-0.1cm}]x1n1.south) -- ([yshift={0.1cm}]xm1.north) ([yshift={-0.1cm}]xm1.south) -- ([yshift={0.1cm}]xmnm.north) ;
				
				\begin{scope}[on background layer]
					\node[rectangle, draw, rounded corners, minimum width=0.7cm, minimum height=4.5cm, fill=gray!20!white, label=below:{\large $K$}] at (0, -1.5) {};
					\node[rectangle, draw, rounded corners, minimum width=0.7cm, minimum height=5.5cm, label=below:{\large $S$}, dashed] at (2, -1.3) {};
				\end{scope}
				
				\node[rectangle, minimum height=0.4cm, text width=5cm, draw=none, fill=none] at (1.5, -5) {\large $F = (K \cup S, E)_{\text{$K$-max}}$};
			\end{scope}
			\begin{scope}[shift={(7,0)}]
				\node[label=above:{\large $a_1$}, fill=gray] (a1) at (0,1) {};
				\node[label=above:{\large $a_2$}, fill=gray] (a2) at (1.5,1) {};
				\node[label={[label distance=-0.1cm]above:{\large $a_{n_1}$}}, fill=gray] (an1) at (2.5,1) {};
				\node[label={[label distance=-0.15cm]above:{\large $a_{n_m}$}}, fill=gray] (anm) at (3.5,1) {};
				\node[label={[label distance=-0.18cm]above:{\large $a_{k-1}$}}, fill=gray] (ak-1) at (4.5,1) {};
				
				\node[label=above:{\large $x^1_1$}, fill=gray] (x11) at (-1,-1) {};
				\node[label=above:{\large $x^1_2$}, fill=gray] (x12) at (0,-1) {};
				\node[label={[label distance=-0.09cm, xshift={-0.4cm}]above:{\large $x^1_{n_1}$}}, fill=gray] (x1n1) at (2,-1) {};
				
				\node[label={[xshift={-0.35cm},label distance=-0.05cm]above:{\large $x^m_1$}}, fill=gray] (xm1) at (3,-1) {};
				\node[label={[label distance=-0.05cm]above:{\large $x^m_2$}}, fill=gray] (xm2) at (3.7,-1) {};
				\node[label={[label distance=-0.15cm]above:{\large $x^m_{n_m}$}}, fill=gray] (xmnm) at (5,-1) {};
				
				\node[label=below:{\large $b_1$}, fill=gray] (b1) at (0,-3) {};
				\node[label=below:{\large $b_2$}, fill=gray] (b2) at (1,-3) {};
				\node[label={[label distance=-0.15cm]below:{\large $b_{m-1}$}}, fill=gray] (bm-1) at (3,-3) {};
				\node[label={[label distance=-0.05cm]below:{\large $b_{m}$}}, fill=gray] (bm) at (4,-3) {};
				
				\draw[thick, dotted] ([xshift={0.1cm}]a2.east) -- ([xshift={-0.1cm}]an1.west) ([xshift={0.1cm}]an1.east) -- ([xshift={-0.1cm}]anm.west)
				([xshift={0.1cm}]anm.east) -- ([xshift={-0.1cm}]ak-1.west) ([xshift={0.1cm}]b2.east) -- ([xshift={-0.1cm}]bm-1.west) ([xshift={0.1cm}]x12.east) -- ([xshift={-0.1cm}]x1n1.west) ([xshift={0.1cm}]x1n1.east) -- ([xshift={-0.1cm}]xm1.west) ([xshift={0.1cm}]xm2.east) -- ([xshift={-0.1cm}]xmnm.west);
				
				\draw[thick] (a1) -- (x11) (a1) -- (xm1) (a2) -- (x12) (a2) -- (xm2) (an1) -- (x1n1) (anm) -- (xmnm);
				
				\draw[thick] (x11) -- (b2) (x11) -- (bm-1) (x11) -- (bm) (x12) -- (b2) (x12) -- (bm-1) (x12) -- (bm) (x1n1) -- (b2) (x1n1) -- (bm-1) (x1n1) -- (bm);
				\draw[thick] (xm1) -- (b1) (xm1) -- (b2) (xm1) -- (bm-1) (xm2) -- (b1) (xm2) -- (b2) (xm2) -- (bm-1) (xmnm) -- (b1) (xmnm) -- (b2) (xmnm) -- (bm-1);
				
				\foreach \i in {1,2,n1,nm,k-1} {
					\draw[thick] (a\i) -- ([yshift={-0.2cm}]a\i.south) (a\i) -- ([xshift={0.2cm},yshift={-0.2cm}]a\i.south);
				}
				
				\begin{scope}[on background layer]
					\node[rectangle, draw, rounded corners, minimum width=5.5cm, minimum height=0.7cm, dashed, label=right:{\large $I$}] at (2.25, 1) {};
					\node[rectangle, draw, rounded corners, minimum width=6.7cm, minimum height=0.7cm, fill=gray!20!white, label=right:{\large $K_n$}] at (2, -1) {};
					\node[rectangle, draw, rounded corners, minimum width=5cm, minimum height=0.7cm, fill=gray!20!white, label=right:{\large $K_m$}] at (2, -3) {};
				\end{scope}
				
				\node[rectangle, minimum height=0.4cm, text width=3cm, draw=none, fill=none] at (3, -4.5) {\large $G$};
			\end{scope}
		\end{tikzpicture}
	\caption{Construction of a graph $G$ such that $F \simeq \mathsf{TS}_k(G)$, where $F$ is a connected split graph satisfying Proposition~\ref{prop:conn-split-isnt-a-TSk-graph}. Vertices of $F$ are labeled by size-$k$ stable sets of $G$. Vertices in a light gray (resp., dashed) box forms a clique (resp., stable set).}
	\label{fig:split-TSk}
\end{figure}
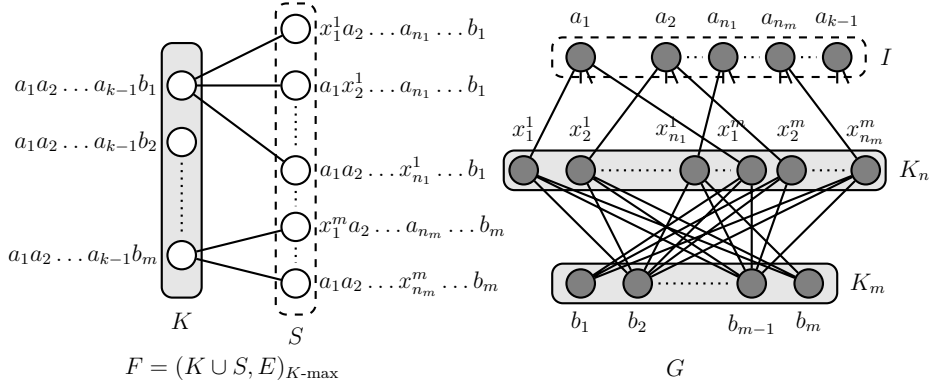
Let $f: V(F) \to V(\mathsf{TS}_k(G))$ be such that $f(v_i) = K^i = I + b_i$ and $f(w^i_j) = K^i - a_j + x^i_j$, for $1 \leq i \leq m$ and $1 \leq j \leq n_i$. 
One can verify that $f$ is well-defined and bijective.
We claim that $F \simeq_f \mathsf{TS}_k(G)$, i.e., $uv \in E(F)$ if and only if $f(u)f(v) \in E(\mathsf{TS}_k(G))$.
\begin{itemize}
	\item[$(\Rightarrow)$] Suppose that $uv \in E(F)$, i.e., either both $u, v$ are in $K$ or $u \in K$ and $v \in S$.
	If $u = v_i \in K$ and $v = v_j \in K$, since $K_m$ is a complete subgraph of $G$, we have $f(u)f(v) = K^iK^j \in E(\mathsf{TS}_k(G))$.
	Otherwise, if $u = v_i \in K$ and $v = w^i_j \in S$, since $a_j$ is the only vertex in $\{a_1, \dots, a_{k-1}\}$ adjacent to $x^i_j$ in $G$, we have $f(u)f(v) = K^i(K^i - a_j + x^i_j) \in E(\mathsf{TS}_k(G))$.
	\item[$(\Leftarrow)$] Suppose that $f(u)f(v) \in E(\mathsf{TS}_k(G))$, i.e., either both $f(u), f(v)$ are in $\bigcup_{i=1}^mK^i$ or $f(u) \in \bigcup_{i=1}^mK^i$ and $f(v) \in \bigcup_{i=1}^m\bigcup_{j=1}^{n_i}(K^i - a_j + x^i_j)$.
	If $f(u) = K^i$ and $f(v) = K^j$, by definition, we have $uv = v_iv_j \in E(F)$.
	Otherwise, if $f(u) = K^i$ and $f(v) \in \bigcup_{i=1}^m\bigcup_{j=1}^{n_i}(K^i - a_j + x^i_j)$, since $f(u)f(v) \in E(\mathsf{TS}_k(G))$, we must have $f(v) = K^i - a_j + x^i_j$ for some $j \in \{1, \dots, n_i\}$, and therefore $uv = v_iw^i_j \in E(F)$.
\end{itemize}

\item[$(\Rightarrow)$] Suppose that either (1) there exists $w \in S$ such that $\vert N_F(w) \vert \geq 2$ or (2) $\vert N_F(w) \vert = 1$ for every $w \in S$ and there exists $v \in K$ such that $\vert N_F(v) \cap S \vert \geq k$.
We claim that $F$ is not a $\mathsf{TS}_k$-reconfiguration graph.
Suppose to the contrary that there exists a graph $G$ such that $F \simeq_f \mathsf{TS}_k(G)$.
Since $F$ is connected, so is $\mathsf{TS}_k(G)$.
From Lemma~\ref{lem:Kn-subgraph}, we must have $f(v_i) = I + b_i$ ($1 \leq i \leq m$) where $\bigcup_{i=1}^m\{b_i\}$ forms a complete subgraph $K_m$ of $G$ and $I = \{a_1, \dots, a_{k-1}\}$ is an independent set of $G - N_G[V(K_m)]$.

If (1) holds, since $\vert K \vert = \omega(F)$, there exists $i \in \{1, \dots, m\}$ such that $v_iw \notin E(F)$, otherwise $K + w$ forms a clique in $F$ of size $\omega(F) + 1$, which is a contradiction.
Lemma~\ref{lem:Kn-subgraph} also implies that $f(w) = I + x$ for some $x \in V(G) - V(K_m) - I$.
Since $v_iw \notin E(F)$, we also have $f(v_i)f(w) = (I+b_i)(I+x) \notin E(\mathsf{TS}_k(G))$, which implies $b_ix \notin E(G)$.
Then, one can verify that $J = (I - a_1) + b_i + x \in V(\mathsf{TS}_k(G))$.
Note that $b_i \notin f(v_j) = I + b_j$ for $j \neq i$ and $1 \leq j \leq m$.
Since $b_ix \notin E(G)$, we have $x \neq b_j$.
Additionally, since $x \notin I$, we must also have $x \notin f(v_j)$.
Since $\mathsf{TS}_k(G)$ is connected, $J$ must be adjacent to some $f(v_j)$ ($1 \leq j \leq m$), which is a contradiction since $\{b_i, x\} \subseteq J - f(v_j)$ for $j \neq i$ (i.e., $J$ is not adjacent to any $f(v_j)$ for $j \neq i$) and $a_1x \notin E(G)$ (i.e., $J$ is not adjacent to $f(v_i) = I + b_i$).

If (2) holds, suppose that $v = v_i$ for some $i \in \{1, \dots, m\}$.
As before, one can show that if $f(w) = I + x$ for some $w \in N_F(v) \cap S$, there must be some contradiction.
Therefore, for every $w \in N_F(v) \cap S$, we must have $f(w) = (I - a_j) + x + b_i$ for some $j \in \{1, \dots, k-1\}$.
Since $\vert N_F(v) \cap S \vert \geq k$, by the pigeonhole principle, there exists $j \in \{1, \dots, k-1\}$ such that $f(v) = f(v_i) = I + b_i$ has two distinct non-adjacent neighbors $J_1 = (I - a_j) + x_1 + b_i$ and $J_2 = (I - a_j) + x_2 + b_i$.
Again, both $x_1$ and $x_2$ are not in $I$, and since none of them are adjacent to $b_i$, they are also not in $\{b_1, \dots, b_m\}$.
This implies $\{x_1, x_2\} \cap f(v_i) = \emptyset$ for $1 \leq i \leq m$.
One can verify that $J = (I - a_j) + x_1 + x_2 \in V(\mathsf{TS}_k(G))$ and it is not adjacent to any $f(v_i)$ ($1 \leq i  \leq m$) since $\{x_1, x_2\} \subseteq J - f(v_i)$.
This contradicts the connectivity of $\mathsf{TS}_k(G)$.
\end{itemize}
\qed\end{proof}

\RA{
We conclude this section with a general result that applies to
all graphs.
$K_{1,3}$ is also known as the \textit{claw} and adding an edge between any two of its degree-$1$ vertices results the \textit{paw}.
$K_4 - e$ is also known as the \textit{diamond}.
(See \figurename~\ref{fig:conn-4}.)
An \textit{outerplanar graph} is a graph that has a planar drawing for which all vertices belong to the outer face of the drawing.
A \textit{maximal outerplanar graph} is an outerplanar graph in which adding a single edge breaks outerplanarity.

\begin{figure}[!ht]
\centering
\begin{adjustbox}{max width=\textwidth}
	\begin{tikzpicture}[every node/.style={draw, thick, circle, fill=white, minimum width=3mm}]
		\begin{scope}
			\foreach \i/\x/\y in {0/0/0,1/1/0,2/1/1,3/0/1} {
				\node (\i) at (\x, \y) {};
			}
			\draw[thick] (2) -- (3) -- (0) -- (1);
			\node[rectangle, draw=none, fill=none] at (0.5,-0.35) {$P_4$};
		\end{scope}
		\begin{scope}[shift={(2,0)}]
			\foreach \i/\x/\y in {0/0/0,1/1/0,2/1/1,3/0/1} {
				\node (\i) at (\x, \y) {};
			}
			\draw[thick] (2) -- (0) (3) -- (0) -- (1);
			\node[rectangle, draw=none, fill=none] at (0.5,-0.35) {claw};
		\end{scope}
		\begin{scope}[shift={(4,0)}]
			\foreach \i/\x/\y in {0/0/0,1/1/0,2/1/1,3/0/1} {
				\node (\i) at (\x, \y) {};
			}
			\draw[thick] (3) -- (0) (2) -- (0) -- (1) -- (2);
			\node[rectangle, draw=none, fill=none] at (0.5,-0.4) {paw};
		\end{scope}
		\begin{scope}[shift={(0,-2)}]
			\foreach \i/\x/\y in {0/0/0,1/1/0,2/1/1,3/0/1} {
				\node (\i) at (\x, \y) {};
			}
			\draw[thick] (2) -- (3) -- (0) -- (1) -- (2);
			\node[rectangle, draw=none, fill=none] at (0.5,-0.35) {$C_4$};
		\end{scope}
		\begin{scope}[shift={(2,-2)}]
			\foreach \i/\x/\y in {0/0/0,1/1/0,2/1/1,3/0/1} {
				\node (\i) at (\x, \y) {};
			}
			\draw[thick] (2) -- (3) -- (0) -- (1) -- (2) -- (0);
			\node[rectangle, draw=none, fill=none] at (0.5,-0.35) {diamond};
		\end{scope}
		\begin{scope}[shift={(4,-2)}]
			\foreach \i/\x/\y in {0/0/0,1/1/0,2/1/1,3/0/1} {
				\node (\i) at (\x, \y) {};
			}
			\draw[thick] (2) -- (3) -- (0) -- (1) -- (2) -- (0) (3) -- (1);
			\node[rectangle, draw=none, fill=none] at (0.5,-0.35) {$K_4$};
		\end{scope}
	\end{tikzpicture}
\end{adjustbox}
\caption{Connected graphs on four vertices.}
\label{fig:conn-4}
\end{figure}
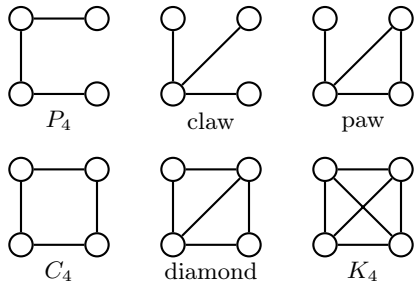

\begin{proposition}\label{prop:general}
\begin{itemize}
	\item[(a)] Let $G$ be a graph whose components are $G_1, G_2, \dots, G_p$, for some $p \geq 1$.
	Suppose that all $G_i$ ($1 \leq i \leq p$) are $\mathsf{TS}_k$-reconfiguration graphs for some fixed integer $k \geq 2$.
	Then, so is $G$.
	\item[(b)] The diamond $K_4 - e$ is the unique smallest graph that is not a $\mathsf{TS}_k$-reconfiguration graph for any $k \geq 2$.
	\item[(c)] The following generalized graphs on $n \geq 4$ vertices of the diamond are not $\mathsf{TS}_k$-reconfiguration graphs for any $k \geq 2$: Maximal outerplanar graphs and $K_n - e$.
	\item[(d)] There exists a disconnected $\mathsf{TS}_k$-reconfiguration graph $G$ such that one of its components is not a $\mathsf{TS}_k$-reconfiguration graph for any $k \geq 2$.
\end{itemize} 
\end{proposition}
\begin{proof}
\begin{itemize}
	\item[(a)] Let $H_i$ ($1 \leq i \leq p$) be such that $G_i = \mathsf{TS}_k(H_i)$.
	Let $H$ be the graph obtained by joining each vertex of $H_i$ to all vertices of $H_j$, for every $1 \leq i < j \leq p$.
	From the above construction, any size-$k$ independent set of $H$ must be also an independent set of $H_i$, for some $i \in \{1, \dots, p\}$.
	As a result, $\mathsf{TS}_k(H)$ is the disjoint union of $\mathsf{TS}_k(H_i)$ for all $i$, that is, $G = \mathsf{TS}_k(H)$.
	
	\item[(b)] Graphs on at most three vertices are disjoint union of paths and cycles, and since any path or cycle is a $\mathsf{TS}_k$-reconfiguration graph for $k \geq 2$ (Corollary~\ref{cor:PnCn-are-TSk-graphs}),  it follows from (a) that so are their disjoint unions.
	If a graph on four vertices is disconnected, each of its components has at most three vertices, and we are done.
	
	Thus, it suffices to consider six connected graphs on four vertices described in \figurename~\ref{fig:conn-4}.
	Among these graphs, $P_4$, $C_4$, and $K_4$ are always $\mathsf{TS}_k$-reconfiguration graphs for $k \geq 2$ and the claw $K_{1,3}$ is a $\mathsf{TS}_k$-reconfiguration graph for $k \geq 3$.
	Let $H$ be the \textit{kite}---the graph obtained from a diamond by attaching a leaf to one of its degree-$2$ vertices.
	One can verify that $\mathsf{TS}_2(H)$ is indeed a paw, and since $\alpha(H) = 2$, it follows from Proposition~\ref{prop:TS-max-reconf-graph} that $\mathsf{TS}_k(H + (k-2)K_1)$ is also a paw, i.e., the paw is always a $\mathsf{TS}_k$-reconfiguration graph for $k \geq 2$.
	
	It remains to show that the diamond is not a $\mathsf{TS}_k$-reconfiguration graph for any $k \geq 2$.
	Suppose to the contrary that there exists a graph $Q$ such that $K_4 - e = \mathsf{TS}_k(Q)$.
	It follows from Lemma~\ref{lem:Kn-subgraph} that vertices of $K_4 - e$ must be of the form $I_1 = I+a$, $I_2 = I+b$, $I_3 = I+c$, $I_4 = I+d$, where $I$ is a size-$(k-1)$ independent set of $Q$ and $\{a, b, c, d\} \subseteq V(Q) - I$ induce a diamond in which, say $a$ and $c$ are non-adjacent. 
	However, for any $x \in I$, note that $J = (I - x) + a + c$ is a size-$k$ independent set of $Q$, and therefore it must be a vertex of $K_4 - e$, which is a contradiction since $J \notin \{I_1, \dots, I_4\}$.
	
	\item[(c)] Use a similar argument as in the proof for the diamond.
	
	\item[(d)] Take $G = \mathsf{TS}_k(\text{diamond} + (k-1)K_1)$. 
	It is not hard to verify that the diamond is a component of $G$ and from (b) it is not a $\mathsf{TS}_k$-reconfiguration graph for any $k \geq 2$.
\end{itemize}
\end{proof}
}

\section{Properties of Reconfiguration Graphs}
\label{sec:properties}
In this section we study various graph properties to see if they are inherited by their %
$\mathsf{TS}$-reconfiguration
graphs, and vice versa.
\RR{
The notation $\mathcal{P}(G) \Rightarrow \mathcal{P}(\mathsf{TS}(G))$ means: if you assume a property $\mathcal{P}$ holds for $G$ (in certain cases, the property always holds, but in some other cases, it may or may not hold), is it true (and under which conditions) that $\mathcal{P}$ also holds for $\mathsf{TS}(G)$?
Conversely, $\mathcal{P}(\mathsf{TS}(G)) \Rightarrow \mathcal{P}(G)$ means that if $\mathcal{P}$ holds for $\mathsf{TS}(G)$ then when does it hold for $G$? For any fixed $k$ we define similar notations $\mathcal{P}(G) \Rightarrow \mathcal{P}(\mathsf{TS}_k(G))$ and $\mathcal{P}(\mathsf{TS}_k(G)) \Rightarrow \mathcal{P}(G)$.
Since $G \simeq \mathsf{TS}_1(G)$, the case $k=1$ is uninteresting and therefore we only consider $k \geq 2$.
}

\RR{
For example, suppose that $G$ is $P_n$ and $\mathcal{P}$ is ``connected''. What we want to know is: if $P_n$ is connected (which is obviously true), is it true (and under which conditions) that $\mathsf{TS}(P_n)$ is also connected? 
This question is trivial for $\mathsf{TS}(P_n)$, because by definition it contains \textit{all} stable sets of $G$ as vertices, and one cannot, for example, join a stable set of size $k$ and a stable set of size $k+1$ by an edge.
This means $\mathsf{TS}(P_n)$ is not connected.
Even for $\mathsf{TS}_k(P_n)$ with some fixed $k$, this question is not much harder, since one can verify that any two size-$k$ stable sets of $P_n$ can be connected by a path in $\mathsf{TS}_k(P_n)$ (just pushing tokens toward one endpoint of the path), which means the graph is clearly connected.
These questions are more challenging with respect to different graphs $G$ and properties $\mathcal{P}$.
Our results are summarized in Table~\ref{table:TS-graph-properties}.
}
\begin{table}[!ht]
\centering
\caption{Some properties of (reconfiguration) graphs. Here $n = \vert V(G) \vert$. There are four cases: (a) $\mathcal{P}(G) \Rightarrow \mathcal{P}(\mathsf{TS}(G))$, (b) $\mathcal{P}(\mathsf{TS}(G)) \Rightarrow \mathcal{P}(G)$, (c) $\mathcal{P}(G) \Rightarrow \mathcal{P}(\mathsf{TS}_k(G))$, and (d) $\mathcal{P}(\mathsf{TS}_k(G)) \Rightarrow \mathcal{P}(G)$.}
\label{table:TS-graph-properties}
\begin{tabular}{|c|c||c|c|c|c|c|}
	\hline
	$\mathcal{P}$ & $G$ & (a) & (b) & (c) & (d) & Ref.\\
	\hline
	$s$-partite & general & \multicolumn{3}{|c|}{yes} & no & Prop.~\ref{prop:s-partite}\\
	\hline
	\multirow{4}{*}{planar} & $P_n$ & \multicolumn{2}{|c|}{yes, iff $n \leq 8$} & \multicolumn{2}{|c|}{\begin{tabular}{@{}c@{}}yes, iff $k = 2$, $n \geq 3$\\ or $k \geq 3$, $n \leq 8$\end{tabular}} & Prop.~\ref{prop:paths-planarity} \\
	\cline{2-7}
	& tree & \multicolumn{4}{|c|}{yes, iff $n \leq 7$} & Prop.~\ref{prop:trees-planarity}\\
	\cline{2-7}
	& $C_n$ & \multicolumn{4}{|c|}{\multirow{2}{*}{yes, iff $n \leq 6$}} & \multirow{2}{*}{Prop.~\ref{prop:cycles-planarity}}\\
	\cline{2-2}
	& connected & \multicolumn{4}{|c|}{} & \\
	\hline
	\multirow{2}{*}{Eulerian} & $C_n$ & no & yes & \multicolumn{2}{|c|}{\RA{yes, iff $1 \leq k < n/2$}} & Prop.~\ref{prop:cycles-Eulerian} \\
	\cline{2-7}
	& general & no & yes & no & no & Prop.~\ref{prop:Eulerian} \\
	\hline
	infinite girth & $P_n$ & \multicolumn{2}{|c|}{\RA{yes, iff $n \leq 4$}} & \multicolumn{2}{|c|}{\RA{yes, iff $n \leq 2k$}} & Prop.~\ref{prop:paths-girth} \\
	\cline{1-7}
	finite girth & $C_n$ & \multicolumn{2}{|c|}{\RA{yes}}& \multicolumn{2}{|c|}{\RA{yes, iff $1 \leq k < n/2$}} & Prop.~\ref{prop:cycles-girth}\\
	\hline
	having $K_s$ & general & \multicolumn{2}{|c|}{yes} & no & yes & Prop.~\ref{prop:cliques-in-graphs}\\
	\hline
\end{tabular}%
\end{table}

\subsection{$s$-Partitedness}\label{sec:s-partitedness}

A \textit{proper $s$-coloring} of a graph $G$ is a mapping $f: V(G) \to \{0, \dots, s-1\}$ such that $f(u) \neq f(v)$ if $uv \in E(G)$, where $u, v \in V(G)$ and $s$ is a positive integer.
If $G$ has a proper $s$-coloring, we call it a \textit{$s$-partite} graph.
The \textit{chromatic number} of a graph $G$, denoted by $\chi(G)$, is the smallest number $s$ such that $G$ has a proper $s$-coloring.

\begin{proposition}\label{prop:s-partite}
\begin{itemize}
\item[(a)] $G$ is $s$-partite if and only if $\mathsf{TS}(G)$ is $s$-partite. In other words, $\chi(G) = \chi(\mathsf{TS}(G))$.
\item[(b)] For each $s \ge 2$, one can construct a graph $G$ such that $\chi(\mathsf{TS}_k(G)) < \chi(G)=s$,
for every $2 \leq k \leq \alpha(G)$.
\end{itemize}
\end{proposition}
\begin{proof}
\begin{itemize}
\item[(a)] (by Masahiro Takahashi.) 
The if direction \RR{follows directly from} $G \simeq \mathsf{TS}_1(G)$.
It remains to show that if $G$ is $s$-partite, so is $\mathsf{TS}(G)$.
Since $G$ is $s$-partite, there exists a proper $s$-coloring $f: V(G) \to \{0, \dots, s-1\}$ of vertices of $G$.
For each independent set $I$ of $G$, let $g(I) = \big( \sum_{v \in I}f(v) \big) \mod s$.
Since $G$ is $s$-partite, each $\mathsf{TS}$-step only slides a token from one color class to another, and therefore if $I$ and $J$ are adjacent in $\mathsf{TS}(G)$, $g(I) \neq g(J)$.
As a result, $g$ is a $s$-coloring of $\mathsf{TS}(G)$, which means $\mathsf{TS}(G)$ is $s$-partite.

To see that $\chi(G) = \chi(\mathsf{TS}(G))$, note that since $G \simeq \mathsf{TS}_1(G)$ is a subgraph of $\mathsf{TS}(G)$, we have $\chi(G) \leq \chi(\mathsf{TS}(G))$.
On the other hand, if $G$ has a $\chi(G)$-coloring, so does $\mathsf{TS}(G)$, i.e., $\chi(\mathsf{TS}(G)) \leq \chi(G)$.

\item[(b)] Fix $s \ge 2$. Let $G^\prime$ be any graph with $\chi(G')=s$ and construct $G$ from $G'$ by adding an additional vertex
$x$ adjacent to all vertices in $V(G')$. We have $\alpha(G) = \alpha(G')$ and since a new colour is required for $x$,  $\chi(G) = \chi(G')+1$. 
For each $2 \leq k \leq \alpha(G)$ the stable sets of size $k$ in $G$ and
$G'$ are identical and so is their adjacency, hence $\mathsf{TS}_k(G) \simeq \mathsf{TS}_k(G')$.
Since $\mathsf{TS}_k(G')$ is a subgraph of $\mathsf{TS}(G')$ we may use part (a) to obtain:
$\chi(\mathsf{TS}_k(G)) = \chi(\mathsf{TS}_k(G')) \le \chi(\mathsf{TS}(G')) =\chi(G')=\chi(G)-1$.
\end{itemize}
\qed\end{proof}

\subsection{Planarity}\label{sec:planarity}

A graph is \textit{planar} if one can draw it in the plane such that its edges intersect only at their endpoints.
The well-known Kuratowski's Theorem says that a graph $G$ is planar if and only if it does not contain any subdivision of $K_5$ or $K_{3,3}$ as a subgraph.

\begin{proposition}\label{prop:paths-planarity}
\begin{itemize}
\item[(a)] $\mathsf{TS}_2(P_n)$ is planar for every $n \geq 3$.
\item[(b)] $\mathsf{TS}_3(P_n)$ is planar for every $n \leq 8$
and non-planar otherwise.
\item[(c)] $\mathsf{TS}(P_n)$ is planar for $n \leq 8$
and non-planar otherwise.
\end{itemize}
\end{proposition}
\begin{proof}
\begin{itemize}
\item[(a)] Suppose that $P_n = v_1\dots v_n$ ($n \geq 3$). 
For each node $\{v_i, v_j\} \in V(\mathsf{TS}_2(P_n))$ ($1 \leq i, j \leq n$ and $|i - j| \geq 2$), its neighbors form a non-empty subset of $\big\{ \{v_{i-1}, v_j\}, \{v_{i+1}, v_j\},\allowbreak \{v_i, v_{j-1}\}, \{v_i, v_{j+1}\} \big\}$. 
Therefore, $\mathsf{TS}_2(P_n)$ can be embedded into a $n \times n$ planar grid graph, which means it is also planar.

\item[(b)] In \figurename~\ref{fig:TS3P8}, we show a planar embedding of $\mathsf{TS}_3(P_8)$.	
In \figurename~\ref{fig:K33-minor-TS3P9} we show a subdivision of $K_{3,3}$ that is contained in $\mathsf{TS}_3(P_9)$.

\item[(c)] $\mathsf{TS}_1(P_8) \simeq P_8$ and it can readily be verified that $\mathsf{TS}_4(P_8) \simeq P_5$. 
From~(a), $\mathsf{TS}_2(P_8)$ is planar. 
From~(b), $\mathsf{TS}_3(P_8)$ is planar and $\mathsf{TS}_3(P_9)$ is not.
\end{itemize}

\begin{figure}[!ht]
\centering
	\begin{tikzpicture}[scale=0.65, every node/.style={draw, thick, circle, transform shape}]
		\foreach \i/\x/\y/\l in {0/31/5/135, 1/31/1/136, 2/20/18/137, 3/20/17/138, 4/14/1/146, 5/28/4/147, 6/20/16/148, 7/25/7/157, 8/20/15/158, 9/20/14/168, 10/17/3/246, 11/17/4/247, 12/14/10/248, 13/17/7/257, 14/16/10/258, 15/20/13/268, 16/20/8/357, 17/18/10/358, 18/19/11/368, 19/21/11/468}
		{
			\node (\i) at (\x, \y) {$\l$};
		}
		\foreach \i/\j in {0/1, 1/2, 1/4, 2/3, 2/5, 3/6, 4/5, 4/10, 5/6, 5/7, 5/11, 6/8, 
			6/12, 7/8, 7/13, 8/9, 8/14, 9/15, 10/11, 11/12, 11/13, 12/14, 13/14, 13/16, 14/15, 14/17, 15/18, 16/17, 17/18, 18/19}
		{
			\draw [thick] (\i) -- (\j);
		}
	\end{tikzpicture}
\caption{A planar drawing of $\mathsf{TS}_3(P_8)$. Each number of the form $abc$ inside a node represents an independent set $\{v_a, v_b, v_c\}$ of $P_8 = v_1\dots v_8$.}
\label{fig:TS3P8}
\end{figure}
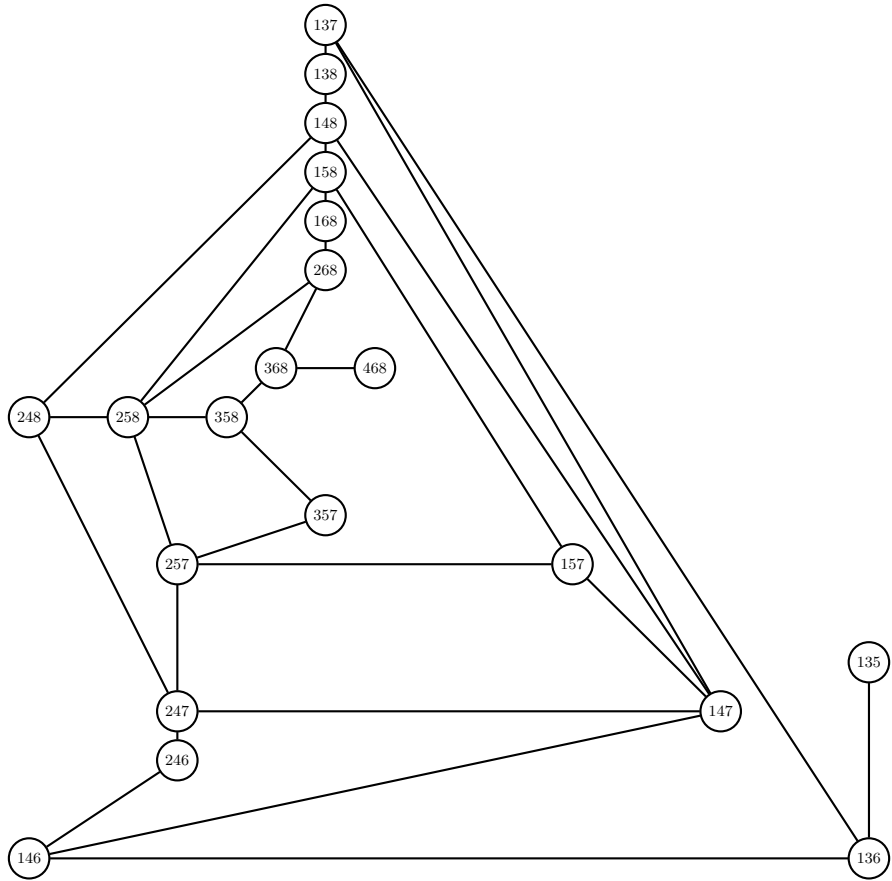

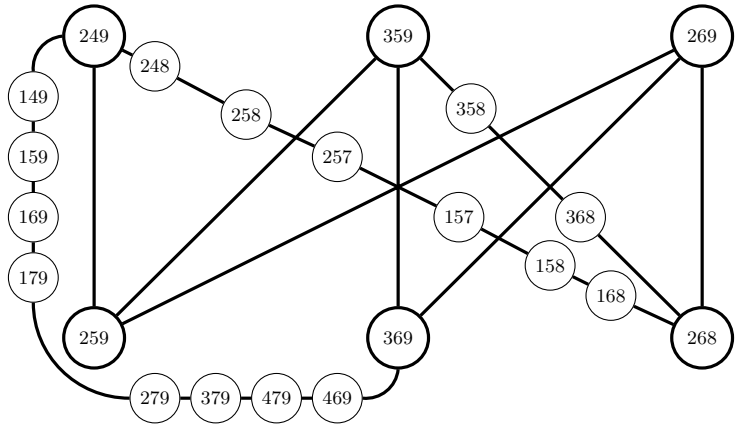
\begin{figure}[!ht]
\centering
	\begin{tikzpicture}[scale=0.8, every node/.style={draw, circle, transform shape}]
		\foreach \x/\y/\l in {0/0/259, 5/0/369, 10/0/268, 0/5/249, 5/5/359, 10/5/269}
		{
			\node [very thick, minimum width=1cm] (\l) at (\x, \y) {$\l$};
		}
		\foreach \x/\y/\l in {6.2/3.8/358, 8/2/368, 1/4.5/248, 2.5/3.7/258, 4/3/257, 6/2/157, 7.5/1.2/158, 8.5/0.7/168, -1/4/149, -1/3/159, -1/2/169, -1/1/179, 1/-1/279, 2/-1/379, 3/-1/479, 4/-1/469}
		{
			\node (\l) at (\x, \y) {$\l$};
		}
		\draw [very thick] (249) -- (259) (259) -- (359) (259) -- (269) (359) -- (369) (369) -- (269) (269) -- (268);
		\draw [very thick] (359) -- (358) -- (368) -- (268);
		\draw [very thick] (249) -- (248) -- (258) -- (257) -- (157) -- (158) -- (168) -- (268);
		\draw [very thick] (249) edge[bend right=45] (149) (149) -- (159) -- (169) -- (179) (179) edge[bend right=45] (279) (279) -- (379) -- (479) -- (469) (469) edge[bend right=45] (369) ;
	\end{tikzpicture}
\caption{A subdivision of $K_{3,3}$ that is contained in $\mathsf{TS}_3(P_9)$. Each number of the form $abc$ inside a node represents an independent set $\{v_a, v_b, v_c\}$ of $P_9 = v_1\dots v_9$.}
\label{fig:K33-minor-TS3P9}
\end{figure}
\qed\end{proof}

\begin{proposition}\label{prop:trees-planarity}
$\mathsf{TS}(T)$ is planar for all trees $T$ having at most seven vertices and non-planar otherwise.
\end{proposition}
\begin{proof}
A computer search showed that:
\begin{itemize}
\item All eleven trees $T$ on seven vertices satisfy $\mathsf{TS}(T)$ is planar.
\item Among twenty-three trees $T$ on eight vertices, seven of them satisfy $\mathsf{TS}(T)$ is non-planar.
\end{itemize}
See \url{https://hoanganhduc.github.io/graphs/} for a list of the mentioned graphs and their corresponding $\mathsf{TS}$-reconfiguration graphs.
\qed\end{proof}

\begin{proposition}\label{prop:cycles-planarity}
\begin{itemize}
\item[(a)] $\mathsf{TS}(C_n)$ is planar for $n \leq 6$
and non-planar otherwise.
\item[(b)] With respect to the number of vertices, $C_7$ is the unique smallest graph $G$ for which $\mathsf{TS}(G)$ is non-planar.
\item[(c)] With respect to the number of edges, including $C_7$, there are eight smallest graphs $G$ for which $\mathsf{TS}(G)$ is non-planar.
\end{itemize}
\end{proposition}
\begin{proof}
A computer search showed that
\begin{itemize}
\item[(i)] For each of the $99$ connected planar graphs $G$ on six vertices, $\mathsf{TS}(G)$ is planar.
\item[(ii)] There are seven trees $T$ on eight vertices whose $\mathsf{TS}(T)$ is non-planar. (Proposition~\ref{prop:trees-planarity}.)
\item[(iii)] $\mathsf{TS}_2(C_7)$ is non-planar. (\figurename~\ref{fig:K33-minor-TS2C7})
\end{itemize}
See \url{https://hoanganhduc.github.io/graphs/} for a list of the mentioned graphs and their corresponding $\mathsf{TS}$-reconfiguration graphs.

\begin{figure}[!ht]
\centering
\begin{tikzpicture}[scale=0.8, every node/.style={draw, circle, transform shape}]
	\foreach \x/\y/\l in {0/0/74, 5/0/52, 10/0/62, 0/5/73, 5/5/63, 10/5/51}
	{
		\node [very thick, minimum width=1cm] (\l) at (\x, \y) {$\l$};
	}
	\foreach \x/\y/\l in {1/4/31, 2/3/41, 4/1/42, 2/4/72, 5/3.5/53, 2/2/64, 4/2/75, 10/2.5/61}
	{
		\node (\l) at (\x, \y) {$\l$};
	}
	
	\draw [very thick] (74) -- (73) (52) -- (51) (62) -- (63);
	\draw [very thick] (73) -- (31) -- (41) -- (42) -- (52);
	\draw [very thick] (73) -- (72) -- (62);
	\draw [very thick] (63) -- (53) -- (52);
	\draw [very thick] (63) -- (64) -- (74);
	\draw [very thick] (51) -- (75) -- (74);
	\draw [very thick] (51) -- (61) -- (62);
\end{tikzpicture}
\caption{A subdivision of $K_{3,3}$ that is contained in $\mathsf{TS}_2(C_7)$. Each number of the form $ab$ inside a node represents an independent set $\{v_a, v_b\}$ of $C_7 = v_1\dots v_7v_1$.}
\label{fig:K33-minor-TS2C7}
\end{figure}
\qed\end{proof}

\begin{corollary}\label{cor:girth-planarity}
\begin{itemize}
\item[(i)] If $\text{girth}(G) \geq 7$, the graph $\mathsf{TS}(G)$ is non-planar.
\item[(ii)] There exists a graph $G$ whose girth is $3$ and $\mathsf{TS}(G)$ is non-planar.
\end{itemize}
\end{corollary}
\begin{proof}
\begin{itemize}
\item[(i)] Let $C$ be a cycle in $G$ whose length is $\text{girth}(G) \geq 7$.
Note that $C$ does not contain any chord; otherwise, we can find a cycle smaller than $C$, which contradicts our assumption.
Now, consider the subgraph $\mathcal{H}$ of $\mathsf{TS}(G)$ induced by only independent sets in $C$.
It follows from Proposition~\ref{prop:cycles-planarity} that $\mathcal{H}$ is non-planar, and so is $\mathsf{TS}(G)$.

\item[(ii)] Take $G$ as the graph obtained by connecting every vertex of $P_n$ to some vertex $v$ not in $P_n$. The subgraph of $\mathsf{TS}(G)$ induced by independent sets in $P_n$ is non-planar when $n \geq 9$ (see Proposition~\ref{prop:paths-planarity}).
\end{itemize}
\qed\end{proof}

\subsection{Eulerianity}
\label{sec:Eulerianity}

A graph $G$ is \textit{Eulerian} if it has an \textit{Eulerian cycle}---a cycle that visits each edge of $G$ exactly once. 
It is well-known that a graph $G$ is Eulerian if and only if it is connected and every vertex has even degree.

\begin{proposition}\label{prop:cycles-Eulerian}
The graph $\mathsf{TS}_k(C_n)$ is Eulerian, for \RA{$1 \leq k < n/2$}.
\end{proposition}
\begin{proof}
Suppose that $C_n = w_1w_2\dots w_n$. 
It is not hard to see that $\mathsf{TS}_k(G)$ is connected, since any independent set of size $k$, where $1 \leq k < n/2$, can be reconfigured to the canonical independent set $\{w_1, w_3,\allowbreak \dots, w_{2k+1}\}$.
Let $I$ be any independent set of $C_n$ of size $k$.
We show that $\deg_{\mathsf{TS}(C_n)}(I)$ is even.
Note that only the maximal odd-length paths $P = v_1v_2\dots v_{2i+1}$ in $C_n$ satisfying $\{v_1, v_3, \dots, v_{2i+1}\} \subseteq I$ affect $\deg_{\mathsf{TS}(C_n)}(I)$, and each of such path contributes either $0$ or $2$ to $\deg_{\mathsf{TS}(C_n)}(I)$.
Thus, $\deg_{\mathsf{TS}(C_n)}(I)$ is even.
\qed\end{proof}

\begin{proposition}\label{prop:Eulerian}
\begin{itemize}
\item[(a)] For any Eulerian graph $G$ on $n \geq 4$ vertices, every component of $\mathsf{TS}_2(G)$ is Eulerian.
\item[(b)] There exists an Eulerian graph $G$ where $\mathsf{TS}_k(G)$ is not Eulerian, for any $k \in \{3, \dots, \alpha(G)\}$.
\item[(c)] For any graph $G$, if $\mathsf{TS}(G)$ is Eulerian, so is $G$. 
Moreover, for any $k \geq 2$, one can construct a graph $G$ such that $G$ is not Eulerian but $\mathsf{TS}_k(G)$ is.
\end{itemize}
\end{proposition}
\begin{proof}
\begin{itemize}
\item[(a)]It suffices to show that all nodes of $\mathsf{TS}_2(G)$ have even degree.
Take any independent set $I = \{v_1, v_2\}$. 
We have $$\deg_{\mathsf{TS}_2(G)}(I) = \deg_G(v_1) + \deg_G(v_2) - 2|N_G(v_1) \cap N_G(v_2)|,$$ which is always even because $G$ is Eulerian.

\item[(b)] %
For $k \geq 3$, take $G$ as the graph containing exactly two cycles: $C_3$ and $C_{2\ell+1} = v_1v_2\dots v_{2\ell+1}$.
Note that $\alpha(G) = \ell + 1$.
Suppose that vertices of $C_{2\ell+1}$ are in counter-clockwise order, $V(C_3) \cap V(C_{2\ell+1}) = \{v_1\}$. 
For $3 \leq k \leq \ell + 1$, let $I = \{w, v_{2\ell+1}, v_2, v_4, \dots, v_{2(k-2)}\}$, where $w \in V(C_3) - v_1$.
In these cases, we all have $\deg_{\mathsf{TS}_k(G)}(I) = 3$, and therefore $\mathsf{TS}_k(G)$ is not Eulerian.

\item[(c)] Clearly, if $\mathsf{TS}(G)$ is Eulerian, it must be connected, and therefore $\mathsf{TS}(G) \simeq \mathsf{TS}_1(G) \simeq G$, which implies $G$ is also Eulerian. 
For $k = 2$, let $G$ be a graph obtained by joining $K_4$ with a single vertex.
Since $G$ has a vertex of degree $1$, it is not Eulerian.
One can verify that $\mathsf{TS}_2(G) \simeq C_3$ and therefore it is Eulerian.
For $k \geq 3$, let $G$ be a graph obtained by joining a vertex of $C_3$ with an endpoint of $P_3$ and with $k-2$ new vertices.
Again, since $G$ has a vertex of degree $1$, it is not Eulerian.
One can verify that $\mathsf{TS}_k(G) \simeq C_4$ and therefore it is Eulerian.
\end{itemize}
\qed\end{proof}

\subsection{Girth}\label{sec:girth}

Recall that the \textit{girth} of a graph $G$ is the smallest size of a cycle in $G$, and is $\infty$ if $G$ is a \textit{forest}, i.e., it has no cycles.

\begin{proposition}\label{prop:paths-girth}

\RA{For every $n \geq 2k+1$, $\text{girth}(\mathsf{TS}_k(P_n))$ is $4$ and $\infty$ otherwise.}
Consequently,
$\text{girth}(\mathsf{TS}(P_n))$ is $4$ for every $n \geq 5$
and $\infty$ otherwise.
\end{proposition}
\RA{
\begin{proof}

From Proposition~\ref{prop:s-partite}, since $\mathsf{TS}(P_n)$ is bipartite, it does not contain any $C_3$, and therefore neither does $\mathsf{TS}_k(P_n)$.
Let $P_{2k+1} = v_1v_2\dots v_{2k+1}$.
One can readily verify that $\mathsf{TS}_k(P_{2k+1})$ has a $C_4$ which contains the size-$k$ independent set $I = \{v_1, v_4, v_7, v_9, \dots, v_{2k+1}\}$.
For $n \geq 2k+1$, since $P_{2k+1}$ is an induced subgraph of $P_n$, Proposition~\ref{prop:induced} implies that $\mathsf{TS}_k(P_{2k+1})$ is also an induced subgraph of $\mathsf{TS}_k(P_n)$, and therefore $\mathsf{TS}_k(P_n)$ also contains a $C_4$.
Thus, $\text{girth}(\mathsf{TS}_k(P_n))$ is $4$ for every $n \geq 2k+1$.
It remains to show that when $n \leq 2k$, the graph $\mathsf{TS}_k(P_n)$ has no cycles.
To see this, note that we also have $k \leq \alpha(P_n) = \lceil n/2 \rceil$.
Therefore, either $k = n/2$ or $k = (n+1)/2$ and in both cases one can verify that $\mathsf{TS}_k(P_n)$ has no cycles.
Consequently, since $\mathsf{TS}_2(P_n)$ is an induced subgraph of $\mathsf{TS}(P_n)$, we have $\text{girth}(\mathsf{TS}(P_n))$ is $4$ for every $n \geq 5 = 2\times 2 + 1$
and $\infty$ otherwise.
\qed\end{proof}
}

\begin{proposition}\label{prop:cycles-girth}
For $1 \leq k < n/2$, $\text{girth}(\mathsf{TS}_k(C_n)) = n$. If $k = n/2$, we have  $\text{girth}(\mathsf{TS}_k(C_n)) = \infty$.
\RA{Consequently, $\text{girth}(\mathsf{TS}(C_n)) = n$.}
\end{proposition}
\begin{proof}
Suppose that $I = \{v_1, v_2, \dots, v_k\}$ is an independent set of $C_n$, where $1 \leq k < n/2$ and $\{v_i\}_{i=1,\dots ,k}$ are ordered in counter-clockwise direction. 
Since no token can ``jump'' above any other tokens, any cycle $\mathcal{C}$ in $\mathsf{TS}(C_n)$ containing $I$ must be formed by moving tokens in counter-clockwise direction such that finally, the token originally placed on $v_i$ is placed on $v_{i+1}$ ($i = 1, 2, \dots, k-1$), and the token originally placed on $v_k$ is placed on $v_1$.
One can achieve this plan with exactly $n$ token-slides, provided that $1 \leq k < n/2$.
If $k = n/2$, the graph $\mathsf{TS}_k(C_n)$ contains exactly two isolated vertices.
Then, its girth is $\infty$.	
\RA{Consequently, since $\mathsf{TS}_k(C_n)$ is a subgraph of $\mathsf{TS}(C_n)$ for $1 \leq k \leq n/2$, it follows that $\text{girth}(\mathsf{TS}(C_n)) = n$.}

\qed\end{proof}

\subsection{Clique of given size}
\label{sec:clique-size}

\begin{proposition}\label{prop:cliques-in-graphs}
\begin{itemize}
\item[(a)] $G$ has a $K_s$ if and only if $\mathsf{TS}(G)$ has a $K_s$ ($s \geq 3$).
\item[(b)] There exists a split graph $G$ such that $G$ has a $K_s$ and $\mathsf{TS}_k(G)$ ($k \geq 2$) does not.
\end{itemize}
\end{proposition}
\begin{proof}
\begin{itemize}
\item[(a)] $G \simeq \mathsf{TS}_1(G)$ implies the only-if direction.
Lemma~\ref{lem:Kn-subgraph} implies the if direction.

\item[(b)] Take a split graph $G = (K \cup S, E)_{\text{$K$-max}}$ such that $\vert K \vert = s$ and there exists $v \in K$ with $\bigcup_{w \in S}N_G(w) = \{v\}$.
Suppose to the contrary that $\mathsf{TS}_k(G)$ ($k \geq 2$) has a $K_s$.
Thus, from Lemma~\ref{lem:Kn-subgraph}, there must be a token $t$ which traverses through all vertices in $K$.
Since $k \geq 2$, there is always at least one token $t^\prime \neq t$ in $S$, which implies that $t$ can never be slid to $v$, a contradiction.
Thus, $\mathsf{TS}_k(G)$ ($k \geq 2$) has no $K_s$.
\end{itemize}
\qed\end{proof}

\section{Decompositions and Joins}\label{sec:decompose-along-join} 
A major problem in constructing the 
$\mathsf{TS}_k$-reconfiguration graph of a given graph $G$ is that
the number of independent sets in $G$ may be exponentially large and hence so may be $\mathsf{TS}_k(G)$. 
However, the structure of $G$
may be such that it can be decomposed into smaller subgraphs for which the
$\mathsf{TS}_k$-reconfiguration graphs
are more easily obtained. This will be useful whenever $\mathsf{TS}_k(G)$ can be built from
these smaller graphs. We present one such decomposition in this section.
\RR{In a subsequent paper \cite{AH2023a} we introduced a different decomposition, called an $H$-decomposition, useful in the context where
the $\mathsf{TS}_k$-reconfiguration graphs are acyclic.}
\begin{definition}
The $H_1,H_2$ join $G$ of vertex disjoint simple graphs $G_1$ and $G_2$, where $H_i \subseteq G_i ,i=1,2$, is formed by adding all edges between $H_1$ and $H_2$.
\end{definition}
\figurename~\ref{fig:my_label}(a) illustrates a join of the graphs $G_1$ and $G_2$. In cases where the graph $G$ can be decomposed along a join we can decompose $\mathsf{TS}_k(G)$ into connected components.
For the decomposition theorem we need
a method of combining stable sets of possibly different sizes from disjoint graphs.
\begin{definition}
Let $G_1$ and $G_2$ be simple connected graphs on different sets of vertices. For
$1 \le s < k$ we define the product $\mathsf{TS}_s(G_1) \otimes  \mathsf{TS}_{k-s}(G_2)$\RR{:} 
The vertices are of the form $S_1 \cup S_2$ where $S_1 \in V(\mathsf{TS}_s(G_1))$
and $S_2 \in V(\mathsf{TS}_{k-s}(G_2))$. The edges in the product are of the form
($S_1 \cup S_2$, $S_1^{'} \cup S_2$) where 
$(S_1,S_1^{'}) \in E(\mathsf{TS}_s(G_1))$ and
($S_1 \cup S_2$, $S_1 \cup S_2^{'}$)
where $(S_2,S_2^{'}) \in E(\mathsf{TS}_{k-s}(G_2))$\RR{.}
\end{definition}
\figurename~\ref{fig:my_label}(b)  shows some examples of this
definition. In the following proposition we show that if graph $G$ is the
join of $G_1$ and $G_2$ then $\mathsf{TS}_{k}(G)$ is the
union of $\mathsf{TS}_{k}(G_1)$, $\mathsf{TS}_{k}(G_2)$ and
products of $\mathsf{TS}_k$-reconfiguration graphs for smaller values of $k$. It
is illustrated in \figurename~\ref{fig:my_label}. 

\begin{figure}[!ht]
\centering
\begin{tikzpicture}[scale=0.64, every node/.style={circle, draw, thick, minimum width=0.3cm, transform shape}]
	\begin{scope}
		\foreach \x/\i/\j in {1/0/0,2/1.5/1,3/1.5/0,4/1.5/-1,5/3/1,6/3/-1,7/4.5/-1,8/4.5/1,9/6/0}
		{
			\node (\x) at (\i,\j) {$\x$};
		}
		\draw[thick] (1) -- (2) -- (5) -- (8) -- (9) -- (7) -- (6) -- (4) -- (1) (1) -- (3) (5) -- (7) (6) -- (8);
		\draw[very thick] (2) -- (5) (2) -- (6) (3) -- (5) (3) -- (6) (4) -- (5) (4) -- (6);
		
		\begin{scope}[on background layer]
			\node[rectangle, draw, label=below:{\large $G_1$}, minimum width=2.5cm, minimum height=4.5cm] at (0.75,0) {};
			\node[rectangle, draw, label=below:{\large $G_2$}, minimum width=4cm, minimum height=4.5cm] at (4.5,0) {};
			\node[rectangle, draw, label=below:{\large $H_1$}, fill=gray!20!white, minimum width=0.7cm, minimum height=3cm] at (1.5,0) {};
			\node[rectangle, draw, label=below:{\large $H_2$}, fill=gray!20!white, minimum width=0.7cm, minimum height=3cm] at (3,0) {};
		\end{scope}
		
		\node[rectangle, draw=none, fill=none] at (3,-3.5) {{\Large (a) $G$ is $H_1, H_2$ join of $G_1, G_2$.}};
	\end{scope}
	\begin{scope}[shift={(10,0)}]
		\node (234) at (0,2) {$234$};
		
		\node (569) at (2,2) {$569$};
		
		\begin{scope}[yshift={-0.5cm}]
			\node (238) at (-2,1) {$238$};
			\node (239) at (-1,0) {$239$};
			\node (237) at (-2,-1) {$237$};
			\draw[thick] (238) -- (239) -- (237);
			
			\node[xshift={1.5cm}] (248) at (-2,1) {$248$};
			\node[xshift={1.5cm}] (249) at (-1,0) {$249$};
			\node[xshift={1.5cm}] (247) at (-2,-1) {$247$};
			\draw[thick] (248) -- (249) -- (247);
			
			\node[xshift={3cm}] (348) at (-2,1) {$348$};
			\node[xshift={3cm}] (349) at (-1,0) {$349$};
			\node[xshift={3cm}] (347) at (-2,-1) {$347$};
			\draw[thick] (348) -- (349) -- (347);
		\end{scope}
		
		\node (178) at (3.5,0.5) {$178$};
		\node (278) at (5,1.5) {$278$};
		\node (378) at (5,0.5) {$378$};
		\node (478) at (5,-0.5) {$478$};
		\draw[thick] (178) -- (278) (178) -- (378) (178) -- (478);
		
		\node (156) at (3.5,2.5) {$156$};
		\node (159) at (3.5,-1.5) {$159$};
		\node (169) at (3.5,-2.5) {$169$};
		
		\begin{scope}[on background layer]
			\node[rectangle, draw, thick, dashed, minimum width=1cm, minimum height=1cm, label={[label distance=-0.5cm]above:{\large $\mathsf{TS}_3(G_1)$}}] at (234) {}; 
			\node[rectangle, draw, thick, dashed, minimum width=1cm, minimum height=1cm, label={[label distance=-0.5cm]above:{\large $\mathsf{TS}_3(G_2)$}}] at (569) {}; 
			\node[rectangle, draw, thick, dashed, minimum width=5cm, minimum height=3cm, label={[label distance=-1.5cm]below:{\large $\mathsf{TS}_2(G_1) \otimes \mathsf{TS}_1(G_2 - H_2)$}}] at ([xshift={-0.5cm}]249) {}; 
			\node[rectangle, draw, thick, dashed, minimum width=1cm, minimum height=6cm, label={[label distance=-1.8cm]above:{\large $\mathsf{TS}_1(G_1 - H_1) \otimes \mathsf{TS}_2(G_2)$}}] at ([yshift={-0.5cm}]178) {}; 
			\node[rectangle, draw, thick, dashed, minimum width=3cm, minimum height=3cm, label={[label distance=-2cm, xshift={2cm}]below:{\large $\mathsf{TS}_1(G_1) \otimes \mathsf{TS}_2(G_2 - H_2)$}}] at ([xshift={-0.7cm}]378) {}; 
		\end{scope}
		
		\node[rectangle, draw=none, fill=none] at (3,-3.5) {{\Large (b) A decomposition of $\mathsf{TS}_k(G)$.}};
	\end{scope}
\end{tikzpicture}
\caption{A graph $G$ and a decomposition of $\mathsf{TS}_k(G)$ using Proposition~\ref{prop:decomp}.}
\label{fig:my_label}
\end{figure}
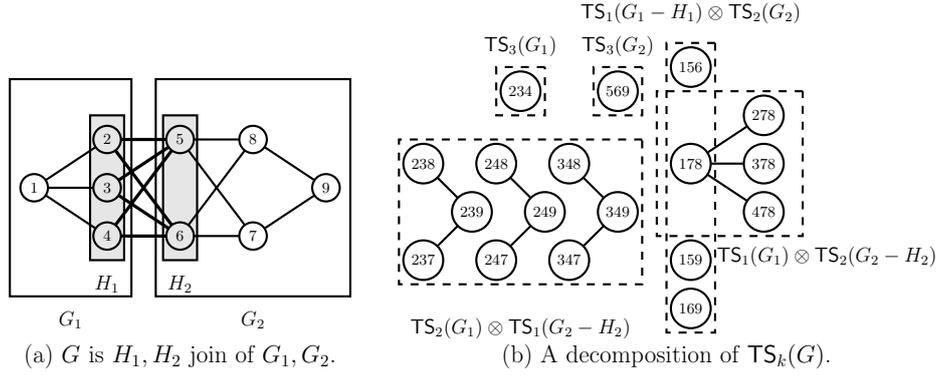

\begin{proposition}
\label{prop:decomp}
Let $G$ be the $H_1,H_2$ join of graphs $G_1$,$G_2$. Choose
an integer $k \ge 2$ so that $G_1$,$G_2$ each contain
a stable set of size $k$.\\
(a) $\mathsf{TS}_k(G)$
be decomposed into $k+1$ (possibly empty) vertex disjoint subgraphs
$\mathsf{TS}_k(G_1)$,  $\mathsf{TS}_k(G_2)$, and for $1 \le s < k$ the union of
\begin{equation}
\label{eq:Ts}
\mathsf{TS}_s(G_1) \otimes  \mathsf{TS}_{k-s}(G_{2} - H_{2}) \text{~~and~~}
\mathsf{TS}_s(G_1 - H_1) \otimes  \mathsf{TS}_{k-s}(G_{2}).
\end{equation}\\
(b) Fix $i=1$ or $2$. Every stable set $S_i$ of size $k$ in $G_i$ satisfies
\begin{equation}
\label{eq:SHcond}
|S_i \cap V(H_i)| \neq 1
\end{equation}
if and only if $\mathsf{TS}_k(G_i)$ is disconnected from the remaining $k$ subgraphs of $\mathsf{TS}_k(G)$. 
\\
\end{proposition}

\begin{proof}
Let $S$ be an independent set of size $k$ in $G$ and, for $i=1,2$
let $S_i = S \cap V(G_i)$. Define $s=|S_1|$. \\
\begin{itemize}
\item[(a)] We consider the three cases.
\begin{itemize}
	\item[(i)] $s=k$.\\
	In this case $S_1$ is an independent set in
	$G_1$, so $S \in V(\mathsf{TS}_k(G_1))$.
	
	\item[(ii)] $s=0$.\\
	The symmetric case, which yields
	$S \in V(\mathsf{TS}_k(G_2))$.
	
	\item[(iii)] $0 < s < k$. \\
	First suppose that $S \cap V(H_1) \neq \emptyset$
	implying, by construction of the join, that
	$S \cap V(H_2) = \emptyset$. We have that
	$S_1 \in V(T_s(G_1)), S_2 \in V(T_{k-s}(G_2 - H_2))$
	and therefore $(S_1,S_2) \in T_s(G_1) \otimes T_{k-s}(G_2 - H_2) $. Symmetrically, if $S \cap V(H_2) \neq \emptyset$ we conclude that $(S_1,S_2) \in T_s(G_1 - H_1) \otimes T_{k-s}(G_2)$. Note that $S$ will appear in
	both sides of the union in (\ref{eq:Ts}) if and only if $S \cap V(H_i) = \emptyset, i=1,2$.
	
	The vertices of the $k+1$ subgraphs of $\mathsf{TS}_k(G)$ in (a) correspond to stable sets in $G$ with different values of $s$, so they are vertex disjoint.
\end{itemize}

\item[(b)] Suppose $i=1$. A similar argument applies to the case $i=2$.
\begin{itemize}
	\item[($\Rightarrow$)] 
	We show that it is impossible to slide a vertex $v \in S_1$ to a vertex $w \in G_2$ to create a new stable set. Indeed, for such a slide to be possible we must have $v \in H_1$ and so by (\ref{eq:Ts}) $S_1 \cap H_1$ 
	must contain another vertex $u$. Now $w \in H_2$ so it is adjacent to $u$ and hence $v$ cannot be slid to $w$. 
	Hence there are no edges from $\mathsf{TS}_k(G_1)$ to any of the other $k$ subgraphs in $\mathsf{TS}_k(G)$.
	
	\item[($\Leftarrow$)] 
	Suppose, by way of contradiction, that there is an edge between $\mathsf{TS}_k(G_1)$ and one of the other $k$ subgraphs. Then there is a vertex $v \in H_1$ that can be slide to
	a vertex $w \in H_2$. This is only possible if
	$|S_1 \cap V(H_1)| =1$ yielding the required contradiction.
\end{itemize}
\end{itemize}
\qed\end{proof}
Observe that equation (\ref{eq:SHcond}) in (b) is automatically 
satisfied if for either $i=1$ or $2$ we have $|V(G_i)| < |V(H_i)| + k -1$.
We remark that even when part (b) applies there may be additional edges between the $k-1$ graphs described by the
products in part (a).

\section{Concluding Remarks}\label{sec:concluding-remarks}

In this paper, \RR{our primary focus was} the realizability and structural properties of reconfiguration graphs of independent sets under Token Sliding.

\RR{In Section~\ref{sec:is-G-a-TSk-graph}}, we presented  necessary and sufficient conditions for a graph $G$ to be a $\mathsf{TS}_k$-reconfiguration graph ($k \geq 2$), where $G$ belongs to certain restricted graph classes, namely complete graphs, paths, cycles, complete bipartite graphs, connected split graphs, \RA{maximal outerplanar graphs, and complete graphs minus one edge}.
Even for $k = 2$, it remains unknown what the necessary and sufficient conditions for a forest to be a $\mathsf{TS}_k$-reconfiguration graph are.
We remark that ``being a $\mathsf{TS}_k$-reconfiguration graph'' is \textit{not} hereditary, even for trees. For example $K_{1,3}$ is not a $\mathsf{TS}_2$-reconfiguration graph (Proposition~\ref{prop:Kmn-isnt-a-TSk-graph}) but if we replace one
edge by a $P_4$ it is (\figurename~\ref{fig:exa-tree-TS2-graph}).
Proposition~\ref{prop:paths-girth} provided a useful insight: given a forest $F$, if there exists $G$ such that $F \simeq \mathsf{TS}_2(G)$, the graph $G$ must be $P_5$-free (e.g., see the graph $G$ in \figurename~\ref{fig:exa-tree-TS2-graph}). 

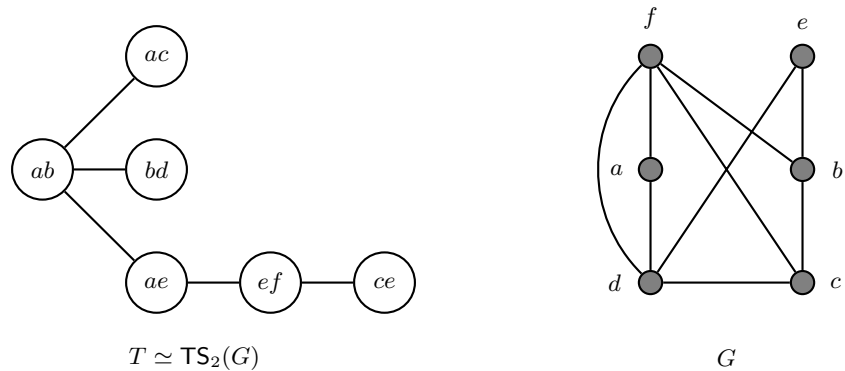
\begin{figure}[!ht]
\centering
\begin{tikzpicture}[every node/.style={circle, draw, thick, minimum size=0.8cm}]
\begin{scope}
\node (v1) at (0,0) {$ab$};
\node (v2) at (1.5,1.5) {$ac$};
\node (v3) at (1.5,0) {$bd$};
\node (v4) at (1.5,-1.5) {$ae$};
\node (v5) at (3,-1.5) {$ef$};
\node (v6) at (4.5,-1.5) {$ce$};

\draw[thick] (v1) -- (v2) (v1) -- (v3) (v1) -- (v4) -- (v5) -- (v6);

\node[draw=none] (T) at (2,-2.5) {$T \simeq \mathsf{TS}_2(G)$};
\end{scope}
\begin{scope}[shift={(8,0)}, every node/.style={circle, draw, thick, fill=gray, minimum size=0.3cm}]
\node[label=left:$a$] (a) at (0,0) {};
\node[label=right:$b$] (b) at (2,0) {};
\node[label=right:$c$] (c) at (2,-1.5) {};
\node[label=left:$d$] (d) at (0,-1.5) {};
\node[label=above:$e$] (e) at (2,1.5) {};
\node[label=above:$f$] (f) at (0,1.5) {};

\draw[thick] (b) -- (c) (a) -- (d) (d) -- (c) (e) -- (b) (e) -- (d) (a) -- (f) (b) -- (f) (c) -- (f) (d) edge[bend left=45] (f);

\node[draw=none, fill=none] (G) at (1,-2.5) {$G$};
\end{scope}
\end{tikzpicture}
\caption{Replacing an edge of $K_{1,3}$ by a $P_4$ results a tree $T$ which is also a $\mathsf{TS}_2$-reconfiguration graph. Each node $ab$ in $T$ represents a size-$2$ stable set of $G$.}
\label{fig:exa-tree-TS2-graph}
\end{figure}

In Section~\ref{sec:properties}, we shifted our focus to some graph properties, namely $s$-partitedness, planarity, Eulerianity, girth, and the clique's size, and provided examples and proofs showing that for some given graph $G$, certain $\mathsf{TS(G)}$ and $\mathsf{TS}_k(G)$ graphs do (not) inherit some properties from $G$, and vice versa. 
As the structural properties of $\mathsf{TS}(G)$ and $\mathsf{TS}_k(G)$ have not yet been systematically investigated, a large number of open questions in this direction can be obtained by strengthening our results or simply specifying either a graph class or a property which has not yet been considered.

In Section~\ref{sec:decompose-along-join} we showed a way of decomposing a graph
that induced a decomposition of its $\mathsf{TS}_k$-reconfiguration graphs.
By inverting the construction we have a way of building larger reconfiguration
graphs from
smaller pieces.

\RA{
In \cite{AH2023a} we continue our study of reconfiguration graphs
of independent sets by focusing on the case where these graphs
are acyclic.}

\section*{Acknowledgements}
We thank Yuni Iwamasa, Jesper Jansson, and Dominik K\"{o}ppl for their useful comments and discussions.
We thank Masahiro Takahashi for his proof of Proposition~\ref{prop:s-partite}(a).
Avis' research is partially supported by the Japan Society for the Promotion of Science (JSPS) KAKENHI Grants JP18H05291, JP20H00579, and JP20H05965 (AFSA) and Hoang's research by JP20H05964 (AFSA).

\bibliographystyle{splncs04}
\bibliography{refs.bib}

\clearpage
\appendix

\section{Connectivity and Diameter of $\mathsf{TS}_k(G)$ for Specific Graph Classes}\label{apd:known-properties}

For certain graph $G$, Table~\ref{table:cd} includes some properties of the connectivity and diameter of $\mathsf{TS}_k(G)$ that can be derived from known results. 

\begin{table}[!ht]
\caption{Connectivity and diameter of $\mathsf{TS}_k(G)$ ($2 \leq k \leq \alpha(G)$). Here $n = \vert V(G) \vert$.}
\label{table:cd}
\centering
\begin{tabular}{|c|c|c||c|c|c|}
\hline
\multicolumn{2}{|c|}{\multirow{2}{*}{$G$ (simple, connected)}} & \multicolumn{2}{|c|}{$\mathsf{TS}_k(G) (1 \leq k \leq \alpha(G))$} & \multirow{2}{*}{Ref.} \\
\cline{3-4}
\multicolumn{2}{|c|}{\quad} & always connected? & diameter & \\
\hline
1 & perfect & no & $O(2^n)$ & \multirow{3}{*}{\cite{KaminskiMM11,KaminskiMM12}}\\
\cline{1-4}
2 & even-hole-free & yes, if $k = \alpha(G)$ & $O(n)$, if $k = \alpha(G)$ & \\
\cline{1-4}
3 & $P_4$-free & no & $O(n^2)$ & \\
\hline
4 & claw-free & yes & $O(\text{poly}(n))$ & \cite{BonsmaKW14}\\
\hline
5 & tree & no & \multirow{2}{*}{$O(n^2)$} & \multirow{2}{*}{\cite{DemaineDFHIOOUY15}}\\
\cline{1-3}
6 & path & yes & & \\
\hline
7 & bipartite permutation & no & $O(n^2)$ & \cite{Fox-EpsteinHOU15}\\
\hline
8 & interval & no & $O(kn^2)$ & \cite{BonamyB17,BrianskiFHM21}\\
\hline
\end{tabular}
\end{table}

\begin{itemize}
\item[1.] In~\cite{KaminskiMM12}, Kami{\'n}ski et al. showed the $\mathtt{PSPACE}$-completeness of \textsc{Independent Set Reconfiguration (ISR)} under any of $\mathsf{TS}, \mathsf{TJ}$, or $\mathsf{TAR}$ when the input graph is a perfect graph.
Combining their reduction from \textsc{Shortest Path Reconfiguration (SPR)} and an example of a reconfiguration graph of \textsc{SPR} having exponentially large diameter in the size of the input graph~\cite{KaminskiMM11} gives us an example of $\mathsf{TS}_k(G)$ with exponentially large diameter in $n = \vert V(G) \vert$.

Observe that one can construct a perfect graph $G$ where $\mathsf{TS}_k(G)$ is not connected.
For instance, take $G$ as the star $K_{1,n}$.
Then, for $n \geq k+1$, $\mathsf{TS}_k(G)$ is not connected. 
This also holds for other graph classes such as $P_4$-free graphs, trees, bipartite permutation graphs, and interval graphs.
\item[2.] They also showed that $\mathsf{TJ}_k(G)$ is connected and its diameter is $O(n)$ when $G$ is a connected even-hole-free graph.
Observe that when $k = \alpha(G)$, we have $\mathsf{TJ}_k(G) \simeq \mathsf{TS}_k(G)$.
\item[3.]
Kami{\'n}ski et al.~\cite{KaminskiMM12} designed a linear-time algorithm that decides whether there is a path between $I, J \in \mathsf{TS}_k(G)$, and if yes, outputs a shortest one, where $G$ is $P_4$-free.
One can verify that their algorithm indeed outputs a path in $\mathsf{TS}_k(G)$ of length $O(n^2)$.
\item[4.] Bonsma et al.~\cite{BonsmaKW14} show that when $G$ is a connected claw-free graph, $\mathsf{TS}_k(G)$ is always connected, and they provided a polynomial-time algorithm for outputting a path between any pair $I, J \in \mathsf{TS}_k(G)$.
\item[5--6.] Demaine et al.~\cite{DemaineDFHIOOUY15} designed a linear-time algorithm for deciding, whether there is a path between $I, J \in \mathsf{TS}_k(G)$, and if yes, output a path of length $O(n^2)$, provided that $G$ is a tree.
They also gave an example of an instance $(G, I, J)$ where $G$ is a path and the length of a shortest path between $I, J \in \mathsf{TS}_k(G)$ is $\Omega(n^2)$. 
\item[7.] Fox-Epstein et al.~\cite{Fox-EpsteinHOU15} designed a cubic-time algorithm for deciding, whether there is a path between $I, J \in \mathsf{TS}_k(G)$, and if yes, output a path of length $O(n^2)$, provided that $G$ is a bipartite permutation graph.
\item[8.] Bonamy and Bousquet~\cite{BonamyB17} designed a polynomial-time algorithm for deciding whether $\mathsf{TS}_k(G)$ is connected when $G$ is an interval graph. 
However, they did not provide any estimation on its diameter.
Motivated by this question, Bria{\'n}ski et al.~\cite{BrianskiFHM21} recently showed that the diameter of $\mathsf{TS}_k(G)$ is $O(kn^2)$.
\end{itemize}

\end{document}